\documentclass{amsproc}

\usepackage{fancyhdr}

\usepackage{url}
\usepackage{color}

\usepackage{eucal}
\usepackage{verbatim}
\usepackage{amsmath}
\usepackage{amssymb}
\usepackage{mathrsfs}
\usepackage{amsthm}
\usepackage{lscape}
\usepackage[cmtip,all]{xy}
\usepackage{nth}

\newtheorem{theorem}{Theorem}[section]
\newtheorem{lemma}[theorem]{Lemma}
\newtheorem{proposition}[theorem]{Proposition}
\newtheorem{corollary}[theorem]{Corollary}

\theoremstyle{definition}
\newtheorem{definition}[theorem]{Definition}
\newtheorem{example}[theorem]{Example}

\theoremstyle{remark}
\newtheorem{remark}[theorem]{Remark}

\numberwithin{equation}{section}

\newcommand{\Z}{\mathbb{Z}}

\newcommand{\mcP}{\mathcal{P}}
\newcommand{\mcM}{\mathcal{M}}

\newcommand{\mcDx}{\mathcal{D}}
\newcommand{\mcDy}{\mathcal{B}}
\newcommand{\mcN}{\mathcal{N}}
\newcommand{\mcC}{\mathcal{C}}

\newcommand{\ra}{\rightarrow}

\newcommand{\kfield}{\mathbf k}

\newcommand{\BC}{{\rm BiCompl}}
\newcommand{\BCv}{{\rm BiCompl}_v}
\newcommand{\antishriek}{\text{!`}}
\newcommand{\As}{{\mathcal{A}s}}
\newcommand{\dAs}{\mathrm{d}\As}

\newcommand{\Coder}{{ \rm Coder}}

\newcommand{\Hom}{{\rm Hom}}
\newcommand{\Mor}{{\rm Mor}}

\newcommand{\sline}{
\vspace{1.5ex}
}

\begin{document}
\title{Representations of derived $A$-infinity algebras}

\author[Aponte Rom\'an]{Camil I. Aponte Rom\'an}
\address{University of Washington\\ Department of Mathematics
Box 354350 Seattle\\ WA 98195-4350}
\email{camili@math.washington.edu}

\author[Livernet]{Muriel Livernet}
\address{Universit\'e Paris 13\\ 
Sorbonne Paris Cit\'e\\ 
LAGA, CNRS, UMR 7539\\ 
 93430  Villetaneuse\\ France}
 \email{livernet@math.univ-paris13.fr}

 \author[Robertson]{Marcy Robertson}
 \address{Department of Mathematics\\
Middlesex College\\
The University of Western Ontario\\
London, Ontario\\
Canada, N6A 5B7}
\email{mrober97@uwo.ca}

\author[Whitehouse]{Sarah Whitehouse}
\address{School of Mathematics and Statistics\\ 
University of Sheffield\\ S3 7RH\\ England}
\email{s.whitehouse@sheffield.ac.uk}

\author[Ziegenhagen]{Stephanie Ziegenhagen}
\address{Fachbereich Mathematik der Universit\"at Hamburg\\
Bundesstrasse 55\\
D-20146 Hamburg\\ Germany}
\email{stephanie.ziegenhagen@uni-hamburg.de}

\begin{abstract}
The notion of a derived A-infinity algebra arose in the work of Sagave as a 
natural generalisation of the classical A-infinity algebra, relevant to the case where one works
over a commutative ring rather than a field.
We develop some of the basic operadic theory of derived $A$-infinity algebras, building on
work of Livernet-Roitzheim-Whitehouse. In particular, we study the coalgebras over the Koszul dual cooperad
of the operad $\dAs$, and provide a simple description of these. We study representations of
derived $A$-infinity algebras and explain how these are a two-sided version of Sagave's modules 
over derived  $A$-infinity algebras. We also give a new explicit example of a derived $A$-infinity algebra.
\end{abstract}

\maketitle

\tableofcontents

\section*{Acknowledgements}
The authors would like to thank the organizers of the Women in Topology workshop in Banff in August 2013
for bringing us together to work on this paper. 

\section{Introduction}
\label{sec:introduction}

Strongly homotopy associative algebras, also known as $A_{\infty}$-algebras, were invented at the beginning of the sixties by Stasheff as a tool in the study of ‘group-like’ topological spaces. Since then it has become clear that $A_{\infty}$-structures are relevant in algebra, geometry and mathematical physics. In particular, Kadeishvili used the existence of $A_{\infty}$-structures in order to classify differential graded algebras over a field up to quasi-isomorphism~\cite{Kad79}.  When the base field is replaced by a commutative ring, however, Kadeishivili's result no longer holds. If the homology of the differential graded algebra is not projective over the ground ring there need no longer be a minimal $A_{\infty}$-algebra quasi-isomorphic to the given differential graded algebra. 

In order to bypass the projectivity assumptions necessary for Kadeishvili's result, Sagave developed the notion of
derived $A_\infty$-algebras~\cite{Sag10}. While classical $A_\infty$-algebras are graded algebras, derived $A_\infty$--algebras are bigraded
algebras. Sagave establishes a notion of minimal model for differential graded algebras (dgas) whose homology is not
necessarily projective by showing that the structure of a derived $A_\infty$--algebra arises on some projective resolution of the
homology of a differential graded algebra.

In this paper, we continue the work of~\cite{LRW13}, developing the description of these structures using operads. The operads we use are non-symmetric operads in the category $\BCv$ of bicomplexes with zero horizontal differential. 
We have an operad $\dAs$ in this category encoding bidgas, which are simply monoids in bicomplexes. 
It is shown in~\cite{LRW13} that derived $A_\infty$-algebras are 
precisely algebras over the operad $$dA_\infty=(\dAs)_\infty= \Omega((\dAs)^{\antishriek}).$$ 
Here $(\dAs)^{\antishriek}$ is the Koszul dual cooperad of the operad  $\dAs$, and $\Omega$ 
denotes the cobar construction.
In this manner, we view a derived  $A_\infty$-algebra as the infinity version
of a bidga,  just as an $A_\infty$-algebra is the infinity version of a dga.
\medskip

We further investigate the operad $\dAs$, in particular studying $(\dAs)^{\antishriek}$-coalgebras. The structure of an $\As^{\antishriek}$-coalgebra is well-known to be equivalent, via a suspension, to
that of a usual coassociative coalgebra. Analogously, $(\dAs)^{\antishriek}$-coalgebras are equivalent, via a suspension in the vertical direction, to coassociative coalgebras which are equipped with an extra piece of structure.

A substantial part of this paper is concerned with representations of derived $A_\infty$-algebras.
Besides being an important part of the basic operadic theory of these algebras, we will use this theory 
in subsequent work to develop the Hochschild cohomology of derived $A_\infty$-algebras with coefficients.
In section~\ref{sec:representations}, we give a general result expressing a representation of a $\mathcal{P}_\infty$-algebra for any Koszul operad $\mathcal{P}$ in terms of 
a square-zero coderivation. Then we work this out explicitly for the derived $A_\infty$ case. 
We explain how this relates to Sagave's derived $A_\infty$-modules: the operadic notion of representation yields a two-sided version of Sagave's modules.

Finally, we present a new, explicit example of a derived $A_\infty$-algebra. The construction
is based on some examples of $A_\infty$-algebras due to Allocca and Lada~\cite{Allocca10}. 
\medskip

The paper is organized as follows. In section~\ref{sec:review} we begin with a brief review of previous work on  derived $A_\infty$-algebras and establish our notation and conventions. Sections~\ref{sec:coalgebras} and~\ref{sec:representations}
cover the material on   $(\dAs)^{\antishriek}$-coalgebras, coderivations and representations.  Section~\ref{sec:examples} presents our new example. A brief appendix establishes the relationship between two standard sign conventions and gives details of cooperadic suspension in our bigraded setting.

\section{Review of derived $A_\infty$-algebras}
\label{sec:review}

In this section we  establish our notation and conventions. We review Sagave's definition of 
derived $A_{\infty}$-algebras from~\cite{Sag10}
and we 
explain the operadic approach of~\cite{LRW13}. 

\subsection{Derived $A_{\infty}$-algebras}

Let $\kfield$ denote a commutative ring unless otherwise stated. We start by considering $(\mathbb{Z},\mathbb{Z})$-bigraded $\kfield$-modules
\[
A = \bigoplus\limits_{i \in \mathbb{Z}, j \in \mathbb{Z}} A^j_i.
\]

We will use the following grading conventions. An element in $A^j_i$ is said to be of bidegree $(i,j)$. 
We call $ i$ the horizontal degree and $j$ the vertical degree.
We have two suspensions:
$$(sA)^j_i = A^{j+1}_i \quad \text{and} \quad (SA)^j_i = A^j_{i+1}.$$
A morphism of bidegree $(u,v)$ maps $A^j_i$ to $A^{j+v}_{i+u}$, hence is a map
of bidegree $(0,0)$
$$s^{-v} S^{-u} A \ra A.$$  

We remark that this is a different convention to that adopted in~\cite{LRW13}.
The difference is a matter of changing the first grading from homological to cohomological conventions.

Note also that our objects are graded over $(\mathbb{Z},\mathbb{Z})$.
The reason for the change will be explained below.

The following definition of (non-unital) derived $A_\infty$-algebra is that of~\cite{Sag10},
except that we generalize to allow a $(\Z,\Z)$-bigrading, rather than an $(\mathbb{N}$,$\mathbb{Z})$-bigrading. 
(Sagave avoids $(\Z,\Z)$-bigrading because of potential problems taking total complexes, but this is not an issue
for the purposes of the present paper.)

\begin{definition}\label{defn:dAinftystr}
A \emph{derived $A_\infty$-algebra} is a $(\mathbb{Z}$,$\mathbb{Z})$-bigraded $\kfield$-module $A$
equipped with $\kfield$-linear maps
\[
m_{ij}: A^{\otimes j} \longrightarrow A
\]
of bidegree $(-i,2-i-j)$ for each $i \ge 0$, $j\ge 1$,  satisfying the equations
\begin{equation}\label{dobjectequation}
\sum\limits_{\substack{ u=i+p, v=j+q-1 \\ j=1+r+t}} (-1)^{rq+t+pj}
m_{ij} (1^{\otimes r} \otimes m_{pq} \otimes 1^{\otimes t}) = 0
\end{equation}
for all $u\ge 0$ and $v \ge 1$. 
\end{definition}
 
Note that the map $m_{ij}$ maps from $(A^{\otimes j})_{\alpha}^{\beta}$ to
 $(A^{\otimes j})_{\alpha-i}^{\beta+2-i-j}$, just as in~\cite{LRW13}. Thus the different convention for bidegrees 
has no effect on signs.

Examples of derived $A_\infty$-algebras include classical $A_\infty$-algebras, which are derived $A_\infty$-algebras concentrated in horizontal degree 0. Other examples 
are bicomplexes, bidgas and twisted chain complexes (see below). 

We remark that we follow the sign conventions of Sagave~\cite{Sag10}. For a derived $A_\infty$-algebra concentrated in horizontal degree 0, one obtains one of the standard sign conventions for $A_\infty$-algebras. The appendix contains
a discussion of alternative sign conventions, with a precise description of the relationship between them.

\subsection{Twisted chain complexes}

The notion of \emph{twisted chain complex} is important in the theory of derived $A_\infty$-algebras.
The term \emph{multicomplex} is also used for a twisted chain complex.

\begin{definition}
\label{def:twistedchaincx}
A \emph{twisted chain complex} $C$ is a $(\mathbb{Z},\mathbb{Z})$-bigraded $\kfield$-module with $\kfield$-linear maps $d_i^C:C\longrightarrow C$ of  bidegree
$(-i, 1-i)$ for $i\geq 0$, satisfying $\sum_{i+p=u} (-1)^{i}d_id_p=0$ for $u\geq 0$. A \emph{map of twisted chain complexes} $C\longrightarrow D$ is 
a family of maps $f_i:C\longrightarrow D$, for $i\geq 0$, of bidegree $(-i,-i)$, satisfying
    $$
    \sum_{i+p=u} (-1)^{i} f_id^C_p = \sum_{i+p=u}  d^D_i f_p.
    $$
The composition of maps $f:E\to F$ and $g: F\to G$ is defined by $(gf)_u=\sum_{i+p=u}g_if_p$ and the resulting category
is denoted ${\rm{tCh}}_k$.
\end{definition}

A derived $A_\infty$-algebra has an underlying twisted chain complex, specified by the maps $m_{i1}$ for $i\geq 0$.

\subsection{Vertical bicomplexes and operads in vertical bicomplexes}

The underlying category for the operadic view of derived $A_\infty$-algebras is the category of
vertical bicomplexes.

\begin{definition}
\label{def:vbicx}
An object of the category of \emph{vertical bicomplexes} $\BC_v$ is a bigraded $\kfield$-module as above equipped with a vertical differential
\[
d_A: A^j_i \longrightarrow A^{j+1}_{i}
\]
of bidegree $(0,1)$. The morphisms are those morphisms of bigraded modules commuting with the vertical differential. We denote by $\Hom(A,B)$ the set of morphisms (preserving the bigrading) from $A$ to $B$.
\end{definition}

The category $\BC_v$ is isomorphic to the category of $\mathbb{Z}$-graded chain complexes of $\kfield$-modules.

For the suspension $s$ as above, we have $d_{sA}(sx)=-s(d_Ax).$

The tensor product of two vertical bicomplexes $A$ and $B$ is given by
    $$
    (A\otimes B)_u^v=\bigoplus_{i+p=u,\, j+q=v}A_i^j\otimes B_p^q,
    $$
with $d_{A\otimes B}=d_A\otimes 1+1\otimes d_B:(A\otimes B)_u^v\to (A\otimes B)_u^{v+1}$.
This makes $\BC_v$ into a symmetric monoidal category.
\medskip

Let  $A$ and $B$ be two vertical bicomplexes. We write $\Hom_\kfield$ for morphisms of $\kfield$-modules. We will denote by $\Mor(A,B)$ the vertical bicomplex given by
 	$$
	\Mor(A,B)_u^{v}=\prod_{\alpha,\beta} \Hom_\kfield(A^\beta_{\alpha},B^{\beta+v}_{\alpha+u}),
	$$ 
with vertical differential given by $\partial_{\Mor}(f)= d_Bf - (-1)^{j}f d_A$ for $f$ of bidegree $(l,j)$.
The reason for the change of grading conventions is that, 
with the convention adopted here,  $\text{Mor}$ is now an internal Hom on  $\text{Bicompl}_v$.
\medskip

The following notation will be useful in applying the Koszul sign rule.
We denote by $|(r,s)||(r',s')|$ the integer $rr'+ss'$.

\subsection{The operad $\dAs$} 

We now describe an operad in $\BC_v$.  All operads considered in this paper are non-symmetric.  

	\begin{definition} 
\label{def:plethysm}
A \emph{collection} in $\BC_v$ is a sequence $A(n)_{n\geq 1}$ where $A(n)$ is a vertical bicomplex for each $n\geq 1$.
We denote by $\mathcal{C}\BC_v$ the category of collections of vertical bicomplexes. This category is endowed with a monoidal structure, the plethysm given by,
\[
(M\circ N)(n)=\bigoplus_{k,\ l_1+\cdots+l_k=n} M(k)\otimes N(l_1)\otimes\cdots\otimes N(l_k),
\]
 for any two collections $M$ and $N$.
The unit for the plethysm is given by the collection 
\[
I(n)=\begin{cases} 0, & \text{ if } n\not=1 \\
\kfield \text { concentrated in bidegree } (0,0), & \text{ if } n=1.\end{cases}
\]
\end{definition}
\begin{definition} A (non-symmetric) \emph{operad} in $\BC_v$ is a monoid in $\mathcal{C}\BC_v$.
\end{definition}

We adopt standard operad notation, so that  $\mathcal{P}(M,R)$ denotes the operad defined by generators
and relations $\mathcal{F}(M)/(R)$, where $\mathcal{F}(M)$ is the free (non-symmetric) operad
on the collection $M$.

\begin{definition}
\label{def:dAs}
The operad $\dAs$ in $\BC_v$ is defined as $\mathcal{P}(M_{\dAs},R_{\dAs})$ where
$$M_{\dAs}(n)=\begin{cases} 0, & \text{if } n>2, \\
\kfield  m_{02} \text{ concentrated in bidegree } (0,0), & \text {if } n=2, \\
\kfield  m_{11} \text{ concentrated in bidegree } (-1,0), & \text {if } n=1, \end{cases}$$
and
$$R_{\dAs}=\kfield (m_{02}\circ_1 m_{02}-m_{02}\circ_2 m_{02})\oplus
\kfield m_{11}^2 \oplus \kfield (m_{11}\circ_1 m_{02}-m_{02}\circ_1 m_{11} - m_{02}\circ_2 m_{11}),$$
with trivial vertical differential.
\end{definition}

The algebras for this operad are easily seen to be the bidgas, that is associative monoids in
bicomplexes; see~\cite[Proposition 2.5]{LRW13}. Note that one differential comes from the vertical differential on objects in the underlying category,
while the operad encodes the other differential and the multiplication.
\medskip

The operad $\dAs$ is Koszul and one of the main results of~\cite{LRW13} identifies the associated infinity
algebras.

\begin{theorem}\cite[Theorem 3.2]{LRW13}
A derived $A_\infty$-algebra is precisely a $(\dAs)_\infty=\Omega((\dAs)^{\antishriek})$-algebra.
\end{theorem}

\section{Coalgebras over the Koszul dual cooperad}
\label{sec:coalgebras}

In this section we initiate a study of the operad $\dAs$ and related objects.
In particular we consider the category of coalgebras over the Koszul dual cooperad of $\dAs$
 and coderivations of such coalgebras. 
This will allow us to give an operadic explanation of Sagave's reformulation of a derived
$A_\infty$-algebra structure in terms of certain structure on the tensor coalgebra.
We begin by setting up cooperads and their coalgebras. Then we
recall the classical case for the associative operad $\As$, before 
considering the derived case.

\subsection{Cooperads and coalgebras}

We briefly set up our conventions for non-symmetric cooperads and (conilpotent) coalgebras over cooperads. 

A non-symmetric cooperad in a monoidal category is a comonoid in the associated category of
collections endowed with the monoidal structure given by plethysm $\circ$ of collections; see Definition~\ref{def:plethysm}. Thus a
non-symmetric cooperad $\mathcal{C}$ has a structure map $\Delta: \mathcal{C}\to \mathcal{C}\circ \mathcal{C}$,
satisfying standard coassociative and counital conditions.

A conilpotent coalgebra $C$ over a cooperad $\mathcal{C}$ has a structure map
	\[
	\Delta_C: C\to \mathcal{C}(C)=\bigoplus_k \mathcal{C}(k)\otimes C^{\otimes k},
	\]
satisfying the standard compatibility with the cooperad structure of $\mathcal{C}$.

\subsection{Cooperadic suspension}\label{subsec:coopsusp}

The notion of suspension of an operad  as in \cite[Section 1.3] {GJ94} can be adapted to collections. 

We define the operation $\Lambda R$ for any collection $R$  in $\BCv$ as follows:

$$\Lambda R(n)=s^{1-n} R(n).$$

If $\mathcal{R}$ is a non-symmetric (co)operad so is $\Lambda{\mathcal{R}}$ and  if ${\mathcal{R}}(V)$ denotes the free (co)algebra (co)generated by $V$ then
$$(\Lambda {\mathcal{R}})(sV)\cong s {\mathcal{R}}(V).$$
Consequently, $V$ is an $\mathcal{R}$-(co)algebra if and only if $sV$ is a $\Lambda \mathcal{R}$-(co)algebra. Equivalently $V$ is a $\Lambda \mathcal{R}$-(co)algebra if and only if
$s^{-1}V$ is an $\mathcal{R}$-(co)algebra. Indeed this construction gives rise to an isomorphism of (co)algebra
categories.

\smallskip

Further details about cooperadic suspension can be found in the appendix, explaining in detail the signs involved
in our bigraded setting.

\subsection{The classical case, $\As^{\antishriek}$-coalgebras}

We denote by  $\As$ the usual operad for associative algebras. This can be viewed either as an operad  in differential 
graded modules, which is the usual classical context, or equivalently in vertical bicomplexes (in which case it is concentrated in horizontal degree zero).
 In the case of this operad, there is a well-known nice relationship, via suspension, between
 the operadic notion of coalgebra over the cooperad  $\As^{\antishriek}$ and ordinary
coassociative coalgebras.

Indeed, since $\Lambda \As^{\antishriek}= \As^*$ we have the following result.

\begin{proposition}
\label{prop:ascoalg}
Cooperadic suspension gives rise to an isomorphism of categories between the
category of conilpotent coalgebras over the cooperad  $\As^{\antishriek}$ 
and the category of conilpotent coassociative coalgebras.

Under this isomorphism the notion of coderivation $d: C \to C$ on a 
coassociative coalgebra $C$ corresponds to the operadic  notion of coderivation 
on the corresponding $\As^{\antishriek}$--coalgebra, $s^{-1}C$. \qed
\end{proposition}

We note that one can remove the conilpotent hypothesis at the expense of using a
completed version of the tensor coalgebra.
\smallskip

Recall that $\As^{\antishriek}(A)=s^{-1}\overline{T}^c(sA)$, the shifted reduced tensor
coalgebra on $sA$.

We can see the basic idea of how the isomorphism works on objects very explicitly:
 given a coassociative coalgebra C with comultiplication 
$\Delta:C \to C \otimes C$, this completely determines an $\As^{\antishriek}$-coalgebra structure on $s^{-1}C$
	\[
	\overline{\Delta}: s^{-1}C \to \As^{\antishriek}(s^{-1}C)=s^{-1}\overline{T}^c(C).
	\]
The components of this map are forced to be (shifted) iterations of the coassociative comultiplication $\Delta$,
that is,  we have $\overline{\Delta}=\sum_{i=0}^\infty s^{-1}\Delta^{(i)}$. (Here we make the
conventions $\Delta^{(0)}=s^{-1}1_C$, $\Delta^{(1)}=s^{-1}\Delta$.)

And, on the other hand, an  $\As^{\antishriek}$-coalgebra structure has to be of this form.
\medskip

Now we have the general theorem that for a suitable operad $\mathcal{P}$, a 
$\mathcal{P}_\infty$-algebra structure on $A$ is equivalent to a square-zero
coderivation of degree one on the $\mathcal{P}^{\antishriek}$-coalgebra $\mathcal{P}^{\antishriek}(A)$; 
see~\cite[10.1.13]{LV12}.
\smallskip

So in the case $\mathcal{P}=\As$, we get that an $\As_\infty=A_\infty$-structure on
$A$ is equivalent to a square-zero coderivation  of degree one on the $\As^{\antishriek}$-coalgebra 
$\As^{\antishriek}(A)=s^{-1}\overline{T}^c(sA)$. And, by the above, this is equivalent to
a square-zero coderivation  of degree one on the coassociative coalgebra $\overline{T}^c(sA)$.

\subsection{The operad of dual numbers}

We recall the situation for the operad of dual numbers, since the operad $\dAs$ can be built 
from the ope\-rad $\As$ and the operad of dual numbers, via a distributive law.

The operad of dual numbers only contains arity one
operations, so it can be thought of as just a $\kfield$-algebra, and algebras over this operad correspond to (left)
modules over this $\kfield$-algebra. So let $\mcDx=\kfield[\epsilon]/(\epsilon^2)$ be the algebra of dual numbers.
We consider this as a bigraded algebra, where the bidegree of $\epsilon$ is $(-1,0)$.

Then consider the Koszul dual cooperad  $\mcDx^{\antishriek}$. Again this is concentrated in
arity one and can be thought of as just a $\kfield$-coalgebra. We have
 $\mcDx^{\antishriek}=\kfield[x]$, where $x=s\epsilon$, $x$ has bidegree
$(-1,-1)$ and the comultiplication is determined by $\Delta (x^n)=\sum_{i+j=n} x^i\otimes x^j$.

A $\mcDx^{\antishriek}$-coalgebra is a comodule $C$ over this coalgebra and this 
turns out to just be a pair $(C,f)$, where $C$ is a  $\kfield$-module
and $f$ is a linear map $f:C\to C$ of bidegree $(1,1)$. 
(Given a coaction $\rho:C\to \mcDx^{\antishriek}\otimes C=\kfield[x]\otimes C$, 
write $f_i$ for the projection onto $\kfield x^i \otimes C$; then coassociativity gives
$f_{m+n}=f_mf_n$, so the coaction is determined by $f_1$.) 
A coderivation is a linear map $d:C\to C$ of bidegree $(r,s)$
such that $df=(-1)^{|(r,s)||(1,1)|}fd$, that is  $df=(-1)^{r+s}fd$. In particular, if $d$ has bidegree
$(0,1)$ then it anti-commutes with $f$.

\subsection{The derived case, $(\dAs)^{\antishriek}$-coalgebras}
\label{subsec:dAsCoalgebras}

\smallskip

We recall from~\cite[Lemma 2.6]{LRW13} that the operad $\dAs$ can be built from the ope\-rad $\As$ and the operad of dual numbers,
via a distributive law, so that we have an isomorphism of operads $\dAs \cong\As\circ\mcDx$.

We have, on underlying collections,  $(\dAs)^{\antishriek}\cong \mcDx^{\antishriek}\circ  (\As)^{\antishriek}$.
Since $\mcDx^{\antishriek}$ is concentrated in arity one, applying $\Lambda$ gives
$\Lambda   (\dAs)^{\antishriek}\cong \mcDx^{\antishriek}\circ  \Lambda (\As)^{\antishriek}$.
It thus seems natural that a  $\Lambda(\dAs)^{\antishriek}$-coalgebra should correspond to a coassociative
coalgebra (coming from the $ \Lambda (\As)^{\antishriek}$-coalgebra structure),
plus a compatible linear map (coming from the $ \mcDx^{\antishriek}$-coalgebra structure). This works out as follows.
\medskip

Consider triples $(C,\Delta, f)$ where $(C, \Delta)$ is a conilpotent coassociative coalgebra
and  $f:C\to C$ is a linear map of bidegree $(1,1)$ satisfying   $(f\otimes 1)\Delta=(1\otimes f)\Delta=\Delta f$.
A morphism between two such triples is a morphism of coalgebras commuting with the given linear maps.

\begin{proposition}
\label{prop:dAscoalg}
Cooperadic suspension gives rise to an isomorphism of categories between the
category of conilpotent coalgebras over the cooperad  $(\dAs)^{\antishriek}$ and the category
of triples $(C,\Delta, f)$ as above.

 An operadic coderivation  of bidegree $(0,1)$ of a  
$(\dAs)^{\antishriek}$-coalgebra $s^{-1}C$ corresponds
on $(C,\Delta, f)$ to a coderivation of bidegree $(0,1)$ of the coalgebra $C$, anti-commuting with the linear map $f$.
\end{proposition}

\begin{proof}
We will see that a triple $(C,\Delta, f)$ as above corresponds to a  $(\dAs)^{\antishriek}$-coalgebra
structure on $s^{-1}C$, or equivalently to a  $\Lambda(\dAs)^{\antishriek}$-coalgebra
structure on $C$.

The cooperad structure of $(\dAs)^{\antishriek}$ is given explicitly in~\cite[Proposition 2.7]{LRW13} 
and the corresponding structure of $\Lambda(\dAs)^{\antishriek}$ is given in the appendix;
 see Corollary~\ref{cor:lambda}. 
In particular, as a $\kfield$-module, it is free on generators $\alpha_{uv}$, with bidegree $(-u,-u)$.

It follows that we can identify $\Lambda(\dAs)^{\antishriek}(C)$ with $\kfield[x]\otimes \overline{T}^c(C)$,
where, for $a\in C^{\otimes v}$,
 $\alpha_{uv}\otimes a \in  \Lambda(\dAs)^{\antishriek}(C)$
is identified with $x^u\otimes a \in \kfield[x]\otimes \overline{T}^c(C)$. 
That is, we have
\[
\Lambda(\dAs)^{\antishriek}(C)=\oplus_v\Lambda(\dAs)^{\antishriek}(v)\otimes C^{\otimes v}
=\oplus_{u,v}\kfield \alpha_{uv}\otimes C^{\otimes v}\cong \oplus_v k[x]\otimes C^{\otimes v}.
\]

Let $C$ be a  coalgebra for the cooperad  $\Lambda(\dAs)^{\antishriek}$, with coaction
	\[
	\rho: C\to \Lambda(\dAs)^{\antishriek}(C)=\kfield[x]\otimes \overline{T}^c(C).
	\]
Write $\rho_{i,j}: C\to  C^{\otimes j}$ for the following composite
\[
\rho_{i,j} \colon \xymatrix{ C \ar[r]^-{\rho} & \kfield[x]\otimes \overline{T}^c(C) \ar@{>>}[r] & 
\kfield x^i \otimes C^{\otimes j} \ar[r]^-{\cong} & C^{\otimes j}}.
\]
Define $\Delta=\rho_{0,2}: C\to  C\otimes C$ and $f=\rho_{1,1}: C\to  C$.
Then, using coassociativity of the coaction and the computation in Example \ref{Ex:signs}, 
one can check that $\Delta$ is coassociative (essentially as in the classical case)
and that 
	\[
	-\rho_{1,2}=(f\otimes 1)\Delta=(1\otimes f)\Delta=\Delta f.
	\]
More generally, one has $\rho_{i,j}=(-1)^{i(j+1)}\Delta^{(j-1)}f^i$. Thus the 
 $\Lambda(\dAs)^{\antishriek}$-coalgebra structure is completely determined by $\Delta$ and $f$.

On the other hand, given a triple $(C,\Delta, f)$ as above, we can define 
 	\[
	\rho_{i,j}=(-1)^{i(j+1)}\Delta^{(j-1)}f^i
	\]
and let $\rho:C\to  (\dAs)^{\antishriek}(C)$ be the
corresponding map. Using the fact that $(f\otimes 1)\Delta=(1\otimes f)\Delta=\Delta f$,
we see that $\rho_{i,j}=(-1)^{i(j+1)}(f^i\otimes 1^{j-i})\Delta^{(j-1)}$ and with this relation
we can check that $\rho$ does make $C$ into a  $\Lambda(\dAs)^{\antishriek}$-coalgebra.

It is straightforward to check the statement about coderivations; we get a coderivation
of the coalgebra as in the classical case, together with compatibility with $f$.
\end{proof}

\begin{example} \label{ex:FreedAsCoalg}
As an example we will compute operadic coderivations of the cofree
$\Lambda(\dAs)^{\antishriek}$-coalgebra cogenerated by $C$. From the proof of Proposition~\ref{prop:dAscoalg} 
we have $\Lambda(\dAs)^{\antishriek}(C)\cong k[x]\otimes \overline T^c(C)$.

From the cooperad structure given in Corollary~\ref{cor:lambda},
the coalgebra structure is given by
	\[
	\Delta(x^i\otimes a_1\otimes \cdots \otimes a_n)
	=\sum_{k=1}^{n-1}\sum_{r+s=i} (-1)^\epsilon 
		(x^r\otimes  a_1\otimes \cdots \otimes a_k)\otimes (x^s\otimes  a_{k+1}\otimes \cdots \otimes a_n),
	\]
where $\epsilon= rn+ik+(s,s)(|a_1|+\cdots + |a_k|)$.

Note that if one denotes by $\pi_0$ the projection of $\kfield[x]\otimes \overline{T}^c(C)$ onto 
$\kfield x^0\otimes\overline{T}^c(C)\cong \overline{T}^c(C)$ then
$$\Delta\pi_0=(\pi_0\otimes\pi_0)\Delta$$ where the first $\Delta$ is the usual deconcatenation product defined on $\overline{T}^c(C)$.

 The linear map
	\[
	f:\kfield[x]\otimes \overline{T}^c(C)\to \kfield[x]\otimes \overline{T}^c(C)
	\] 
is  determined by $f(x^n\otimes a)=(-1)^{j+1} x^{n-1}\otimes a$, for $a\in C^{\otimes j}$.
\smallskip

Now an operadic coderivation of bidegree $(0,1)$ is  a coderivation
of the coalgebra $\kfield[x]\otimes \overline{T}^c(C)$,
anti-commuting with the map $f$.
Let $d:\kfield[x]\otimes \overline{T}^c(C) \to \kfield[x]\otimes \overline{T}^c(C)$ and write
	\[
	d(x^n\otimes a)=\sum_{i} x^i\otimes d^{n,i}(a),
	\] 
where $d^{n,i}: \overline{T}^c(C)\to \overline{T}^c(C)$ and $a\in C^{\otimes j}$. 

Write $d^{n,i}(a)=\sum_k d^{n,i,k}(a)$ with $ d^{n,i,k}(a)\in C^{\otimes k}$.
Then anti-commuting with $f$ means that
	\[
	d^{n,i,k}(a)=(-1)^{j+k+1} d^{n-1, i-1,k}(a),
	\] 
where $a\in C^{\otimes j}$ and hence that
	\[
	d^{n,i,k}(a)=(-1)^{i(j+k+1)} d^{n-i,0,k}(a)
	\]
for $i\leq n$ and $d^{n,i,k}=0$ for $i>n$. 
So $d$ is completely determined by the family of maps $d^{n,0,k}$.

Define $\delta^n:  \overline{T}^c(C) \to  \overline{T}^c(C)$ 
by $\delta^n(a)=(-1)^{nj}d^{n,0}(a)=(-1)^{nj}\pi_0d(x^n\otimes a)$,  where $a\in C^{\otimes j}$.

The coderivation condition for $d$ makes each $\delta^{n}$ a coderivation of  $\overline{T}^c(C)$. 
Indeed one can check that for $a\in C^{\otimes j}$,
\begin{multline*}
\Delta\delta^n(a)=(-1)^{nj}\Delta\pi_0 d(x^n\otimes a)=\\
(-1)^{nj}\pi_0\otimes\pi_0(d\otimes 1+1\otimes d)
\Delta(x^n\otimes a)=(\delta^n\otimes 1+1\otimes\delta^n)\Delta(a).
\end{multline*}

So
 we obtain a family of coderivations $\delta^n$ on $\overline{T}^c(C)$ of bidegree $(-n,1-n)$.
\end{example}

Using this we have  an operadic explanation of the following formulation
of a derived $A_\infty$-algebra structure; this is part of~\cite[Lemma 4.1]{Sag10}.

\begin{proposition}\label{prop:dAInfAlgAsTw}
A derived $A_\infty$-algebra structure on a bigraded $\kfield$-module $A$ is equivalent to
specifying a family of coderivations $\overline{T}^c(sA)\to  \overline{T}^c(sA)$ making
$\overline{T}^c(sA)$ into a twisted chain complex.
\end{proposition}

\begin{proof}
As recalled above, a 
$\mathcal{P}_\infty$-algebra structure on $A$ is equivalent to a square-zero
coderivation on the $\mathcal{P}^{\antishriek}$-coalgebra $\mathcal{P}^{\antishriek}(A)$.
Applying this to the example $\mathcal{P}=\dAs$, and with $A=s^{-1}C$, 
we see that a coderivation $d$
of  $(\dAs)^{\antishriek}(A)$ corresponds
to a family of coderivations $\delta^{n}$ on $\overline{T}^c(sA)$ of bidegree $(-n,1-n)$.

Now one can check that if we further impose the condition $d^2=0$ on the
map  $\mathcal{P}^{\antishriek}(A)\to  \mathcal{P}^{\antishriek}(A)$, this corresponds to saying that 
the maps $\delta^{n}$ make $\overline{T}^c(sA)$ into
a twisted chain complex.

In more detail, with $a\in C^{\otimes j}$ and using the same notation as in Example~\ref{ex:FreedAsCoalg},
	\[
	d^2(x^n\otimes a)=\sum_{r,s} x^s\otimes d^{r,s}d^{n,r}(a).
	\]
In particular, considering $s=0$, we see that $d^2=0$ implies:
	\begin{align*}
	\sum_r d^{r,0}d^{n,r}(a)=0 
		&\Leftrightarrow \sum_r\sum_k (-1)^{r(j+k+1)} d^{r,0}d^{n-r,0,k}(a)=0\\
		&\Leftrightarrow \sum_r \sum_k (-1)^{r(j+k+1)+rk+(n-r)j} \delta^r \delta^{n-r,k}(a)=0\\
		&\Leftrightarrow \sum_r (-1)^{r+nj} \delta^r \delta^{n-r}(a)=0\\
		&\Leftrightarrow (-1)^{nj} \sum_r  (-1)^r  \delta^r\delta^{n-r}(a)=0.\\
	\end{align*}
 
Thus $d^2=0$ implies the twisted chain complex conditions $\sum_r  (-1)^r  \delta^r\delta^{n-r}(a)=0$ on the
maps $\delta^r$.

Furthermore, by~\cite[6.3.8]{LV12}, $d^2$ is completely determined by its projection to 
$A$ included in the $x^{0}$ part and it follows that the condition $d^2=0$ holds if and only if
the maps $\delta^r$ satisfy the twisted chain complex conditions.
\end{proof}

\section{Representations of derived $A_{\infty}$-algebras}
\label{sec:representations}
The aim of this section is to study representations of $\dAs_{\infty}$-algebras. We establish some general results on coderivations of representations of coalgebras and then show that representations of 
homotopy algebras correspond to square-zero coderivations on a certain cofree object. We then use these results to describe $\dAs_{\infty}$-representations in terms of a twisted chain complex of coderivations on the tensor coalgebra. 
Thus we obtain a description of $\dAs_{\infty}$-representations similar in spirit to Proposition~\ref{prop:dAInfAlgAsTw}.

\subsection{Coderivations on representations of coalgebras}

One way to describe $\mcP_{\infty}$-structures is via coderivations on cofree coalgebras. We will see that analogously $\mcP_{\infty}$-representations can be described via coderivations on cofree representations of coalgebras, which we will introduce now.  We work in the category $\BCv$ of vertical bicomplexes.

\begin{definition} \label{defn:InfCompProd}
Let $X$ and $Y$ be vertical bicomplexes and let $\mcM$ be a collection in $\BCv$. The vertical bicomplex $\mcM(X;Y)$ is given by 
$$\mcM(X;Y) = \bigoplus_{n\geq 1} \mcM(n) \otimes \bigl(\bigoplus_{a+b+1=n} X^{\otimes a} \otimes Y \otimes X^{\otimes b}\bigr).$$
If $f\colon \mcM \ra \mcM'$ is a map of collections and $g\colon X \ra X'$ and $h\colon Y \ra Y'$ are maps of vertical bicomplexes, the map 
$$f(g;h)\colon \mcM(X;Y) \ra \mcM'(X';Y')$$
 is defined as the direct sum of the maps $ f(a+b+1) \otimes g^{\otimes a} \otimes h \otimes g^{\otimes b}$.
\end{definition}

\begin{remark}\rm
In this section for convenience we drop the symbol $\circ$ for plethysm of collections and
just write $\mathcal{C}\mathcal{C}$ for $\mathcal{C}\circ \mathcal{C}$.

One has to be careful when working with $\mcM(X;Y)$. For example if $\mcN$ is another collection, in general
$$(\mcM \mcN)(X;Y) \ncong \mcM (\mcN(X;Y)).$$
However it is true that $(\mcM \mcN)(X;Y) \cong \mcM(\mcN(X); \mcN(X;Y))$ and we will make frequent use of this.
\end{remark}

Dual to the notion of representation (see e.g.\ \cite{Fr09}) of an algebra over an operad is the notion of representation of a coalgebra over a cooperad. In the following let $(\mcC, \Delta, \epsilon)$ be a cooperad and let $C$ be a $\mcC$-coalgebra with coalgebra structure map $\rho \colon C \ra \mcC (C)$.

\begin{definition} A bigraded module $E$ is called a \emph{representation of $C$ over $\mcC$} if there is a map 
$$\omega \colon E \ra \mcC(C;E)$$
such that the diagrams
$$\xymatrix{
E \ar[r]^-{\omega} \ar[d]_-{\omega} & \mcC(C;E) \ar[d]_-{\Delta}\\
\mcC(C;E) \ar[r]^-{\mcC(\rho; \omega)} & \mcC( \mcC(C); \mcC(C;E)) \cong (\mcC \mcC)(C;E)
}
 \quad \text{and} \quad 
\xymatrix{
E \ar[r]^-{\omega} \ar@{=}[rd] & \mcC(C;E) \ar[d]_-{\epsilon}\\
& E
}
$$
commute.
\end{definition}

\begin{example}
The example we will be primarily interested in is the following cofree representation. Let $C= \mcC(X)$ be the cofree 
$\mcC$-coalgebra cogenerated by $X$. Then to a bigraded module $Y$ we can associate the 
representation $\mcC(X;Y)$. The structure map is given by the comultiplication on $\mcC$, i.e.\ 
$$ \mcC(X; Y ) \ra (\mcC \mcC)(X;Y) \cong \mcC( \mcC(X); \mcC(X;Y)).$$
Over an arbitrary $\mcC$-coalgebra, cofree representations are not that simple, 
see for instance the result on free representations in \cite[4.3.2]{Fr09}.
\end{example}

\begin{remark}\rm
 In \cite[4.3]{Fr09}, Fresse defines the enveloping algebra $U_{\mathcal{P}}(A)$ of an algebra $A$ over an operad $\mathcal P$, so that left modules over
 $U_{\mathcal P}(A)$ are precisely representations of the $\mathcal P$-algebra $A$. This enveloping algebra is obtained as the space of unary operations
 of the enveloping operad. In \cite{Yalin13}, Yalin defines the notion of an enveloping cooperad associated to a coalgebra $C$ over a cooperad $\mathcal C$.
 Similarly to Fresse, if one takes the space of unary operations of this enveloping cooperad one gets the enveloping coalgebra $U_{\mcC}(C)$ so that
 left comodules over $U_{\mcC}(C)$ are precisely representations of the $\mcC$-coalgebra $C$. However, the constructions of Fresse and Yalin are performed 
 in the category of symmetric (co)operads. In this paper we are dealing with non-symmetric (co)operads. But, 
 the constructions of Yalin can be adapted to our case if one considers first the collection
 $$\mcC[C](n)=\bigoplus_{r\geq 0} \mcC(n+r)\otimes (C\oplus\kfield \rho)^{n+r}_n$$
 where $\rho$ is a ``silent'' variable of weight $1$ 
 and $(C\oplus\kfield \rho)^{n+r}_n$ is the component of $(C\oplus\kfield \rho)^{n+r}$ of weight $n$. This collection forms a cooperad and the enveloping cooperad
 is a subcooperad of $\mcC[C]$ obtained as an equalizer like in \cite[2.2]{Yalin13}.
\end{remark}

Next we will define what a coderivation of a representation is. To do this we need to extend the infinitesimal composite $\circ'$ of maps as defined in \cite[6.1.3]{LV12}.

\begin{definition}
Let $\mcM$, $X$ and $Y$ be as in Definition~\ref{defn:InfCompProd}. For $g\colon X \ra X$ and $h\colon Y \ra Y$ the map 
$$\mcM \circ''(g;h) \colon \mcM(X;Y) \ra \mcM(X;Y)$$ 
is defined on $\mcM(a+b+1) \otimes X^{\otimes a} \otimes Y \otimes X^{\otimes b}$ as the sum
$$ \sum_{\stackrel{i=1,}{ i\neq a+1}}^{a+b+1} \mcM \otimes 1^{\otimes i-1} \otimes g \otimes 1^{\otimes a+b+1-i} + \mcM \otimes 1^{\otimes a} \otimes h \otimes 1^{\otimes b} 
$$
with $1$ denoting either the identity on $X$ or $Y$.
\end{definition}

Let $d_{\mcC}$ denote the (vertical) differential of the cooperad $\mcC$, $(C, \rho)$ a $\mcC$-coalgebra in bigraded modules equipped with a coderivation $\partial_C$ and $(E,\omega)$ a bigraded module equipped with a map $\omega$ making it a representation of $C$.
\begin{definition} 
A map $g\colon E \ra E$ is called a coderivation if

$$\xymatrix{
 E \ar[rr]^-{g} \ar[d]_{\omega} &&  E \ar[d]_{\omega} \\
\mcC(C;E) \ar[rr]^*!/^2.5pt/{\labelstyle \mcC \circ'' (\partial_C;g) + d_{ \mcC}(C;E)} 
&& \mcC(C;E)
}$$

commutes.
\end{definition}

We will need analogues of well known  results for coderivations on coalgebras. To simplify formulas we encode coderivations via a distributive law; see \cite{Beck69}. 

\begin{definition}
Let $(\mcP, \gamma, \eta)$ be an operad and $(\mcC,  \Delta, \epsilon)$ a cooperad. A mixed distributive law is a morphism of collections
$$\beta  \colon \mcP \mcC \ra \mcC \mcP$$
such that the diagrams
$$\xymatrix{
\mcP \mcP \mcC \ar[rr]^-{\gamma \mcC} \ar[d]_-{\mcP \beta} && \mcP \mcC \ar[d]_-{\beta}\\
\mcP \mcC \mcP \ar[r]^-{\beta \mcP} &\mcC \mcP \mcP \ar[r]^-{\mcC \gamma} & \mcC \mcP
}
\quad\quad
\xymatrix{
\mcP \mcC \ar[r]^-{\mcP \Delta} \ar[d]_{\beta} & \mcP \mcC \mcC \ar[r]^-{\beta \mcC} & \mcC \mcP \mcC \ar[d]_-{\mcC \beta}\\
\mcC \mcP \ar[rr]^-{\Delta \mcP} && \mcC \mcC \mcP
}$$
$$
\xymatrix{
\mcC \ar[rd]_-{\mcC \eta}\ar[r]^-{\eta \mcC} & \mcP \mcC \ar[d]_-{\beta}\\
& \mcC \mcP
}
\quad\quad
\xymatrix{
\mcP \mcC \ar[rd]^-{\mcP\epsilon} \ar[d]_{\beta}& \\
\mcC \mcP \ar[r]^-{\epsilon \mcP}& \mcP
}
$$
commute.
\end{definition}

The operad $(\mcDy, \gamma_{\mcDy}, \eta_{\mcDy})$ that will help us to describe 
coderivations is the operad freely generated by a single unary operation $y$. In all of our examples $y$ 
will be of bidegree $(0,1)$.

\begin{definition}
We define a distributive law 
$$\beta \colon \mcDy\mcC \ra \mcC \mcDy$$
by requiring that
$$\beta(y;c) = \sum_{i=1}^{n} (-1)^{\vert c \vert \vert y \vert}  c;1^{\otimes i-1}\otimes y 
\otimes 1^{\otimes n-i} + d_{\mcC}(c); 1^{\otimes n}$$
for $c \in \mcC(n)$.
\end{definition}

Since $\mcDy$ is freely generated  we need to check  only that
$$\xymatrix{
\kfield y \otimes \mcC \subset \mcDy  \mcC \ar[d]_-{\beta} \ar[r]^-{\mcDy \Delta} & \mcDy \mcC   
\mcC \ar[r]^-{\beta\mcC} & \mcC \mcDy \mcC   \ar[d]_-{\mcC \beta}\\
 \mcC \mcDy \ar[rr]^-{\Delta \mcDy}&  & \mcC \mcC \mcDy  
}$$
commutes and that $\epsilon \mcDy(\beta(y;c))= y(\epsilon c)$, which can be easily calculated. The other 
two defining conditions of a mixed distributive law determine $\beta$ on the whole of $\mcDy \mcC$. 

It is possible to characterise coderivations via $\beta$. Since a coderivation on a representation depends on the coderivation on the coalgebra we state the corresponding result for coalgebras.

\begin{proposition}\label{prop:CoderViaDCoalg}
Giving a coderivation on a $\mcC$-coalgebra $(C, \rho)$ is equivalent to defining a $\mcDy$-algebra structure $\gamma_C$ on $C$ such that
$$
\xymatrix{
\mcDy (C) \ar[r]^-{\mcDy \rho} \ar[dd]_-{\gamma_C}&   \mcDy\mcC(C) \ar[d]_-{\beta C} \\
&\mcC\mcDy (C) \ar[d]_-{\mcC \gamma_C} \\
C \ar[r]^{\rho}  &  \mcC(C)
}
$$
commutes. Explicitly, the coderivation defined by $\gamma_C$ is $\gamma_C(y) $.
\end{proposition}

We omit the proof of this proposition since it is analogous to the proof of the result for representations which we will state and prove now. 
Again let $\gamma_C \colon \mcDy C \ra C$ correspond to the coderivation $\partial_C$.
 
Observe that since $\mcDy$ is concentrated in arity one we have
$$(\mcDy \mcC)(C;E) \cong \mcDy (\mcC(C;E)) \quad \text{as well as} \quad (\mcC \mcDy)(C;E) \cong \mcC(\mcDy(C); \mcDy(E)).$$

\begin{proposition} \label{prop:CoderRepViaD}
Giving a coderivation on $E$ is equivalent to giving a $\mcDy$-algebra structure map $ \gamma_E \colon \mcDy (E) \ra E$
such that
\begin{equation} \label{CoderRepViaD}
\xymatrix{
\mcDy (E)\ar[r]^-{\mcDy \omega} \ar[dd]_-{\gamma_E} & \mcDy (\mcC (C;E)) = (\mcDy \mcC)(C;E) \ar[d]_-{\beta (C;E)}\\ 
  & (\mcC \mcDy)(C;E) \cong \mcC(\mcDy(C);\mcDy(E)) \ar[d]_-{\mcC(\gamma_C; \gamma_E)}\\
E \ar[r]^-{\omega} & \mcC(C;E)
}
\end{equation}
commutes. The coderivation defined by $\gamma_E$ is $\gamma_E(y)$.
\end{proposition}
\begin{proof}
Since $\mcDy$ is free as an operad generated by $y$, making $E$ a $\mcDy$-algebra is equivalent to specifying $\gamma_E(y)$.
Observe that the condition that the diagram commutes is trivial when we restrict to $I E \subset \mcDy E$.
On the other hand, one easily checks that restricted to $\kfield y \otimes E $ the diagram expresses exactly that $g=\gamma_E(y)$ is a coderivation: the left hand side composition of the maps in the diagram then equals $\omega g$, while the right hand side equals $(\mcC \circ''(\partial_C;g)) \omega+ d_{\mcC}\omega$. 
To show that this implies the general case we proceed by induction. Suppose that (\ref{CoderRepViaD}) holds restricted to $\mcDy_n$ as well as 
restricted to $\mcDy_m$, where $\mcDy_n$ is the sub-$\kfield$-module of $\mcDy$ spanned by  $\{y^i, i\leq n\}$. We need to show that
$$\xymatrix{
\mcDy_n \mcDy_m (E) \ar[d]_-{\gamma_{\mcDy}E} \ar[r]^-{\gamma_{\mcDy}E } & \mcDy_{n+m} (E) \ar[d]_-{\gamma_E} \\
\mcDy_{n+m}(E) \ar[d]_-{\mcDy \omega} & E \ar[dd]_-{\omega}\\
\mcDy_{n+m}( \mcC(C;E)) = (\mcDy_{n+m} \mcC)(C;E) \ar[d]_-{\beta(C;E)} &\\
(\mcC \mcDy_{n+m})(C;E) \cong \mcC(\mcDy_{n+m}(C); \mcDy_{n+m}(E)) \ar[r]^-{\mcC (\gamma_C; \gamma_E)}  & \mcC(C;E) 
}$$
commutes. 
Keep in mind that $\gamma_E$ defines an algebra structure and note that we have the identities
$$ (\mcDy \omega)(\gamma_{\mcDy} E) = (\gamma_{\mcDy} \mcC(C;E))(\mcDy \mcDy \omega)$$
and
$$\beta (\gamma_{\mcDy} \mcC) = (\mcC \gamma_{\mcDy})( \beta \mcDy)( \mcDy \beta).$$
Then using that (\ref{CoderRepViaD}) holds restricted to $\mcDy_m$ and $\mcDy_n$ we  find that the right and the upper square in the diagram
$$\scalebox{0.90}[0.90]{
\xymatrix{
\mcDy_n \mcDy_m (E) \ar[d]_-{\mcDy \mcDy \omega} \ar[r]^-{\mcDy \gamma_E} & \mcDy_{n}  (E) \ar[dd]_-{\mcDy \omega} \ar[r]^-{\gamma_E} & E \ar[dddd]^-{\omega}\\
\mcDy_{n} \mcDy_{m}\mcC(C;E)  \ar[d]_-{\mcDy\beta(C;E)} &&\\
\mcDy_{n} (\mcC \mcDy_m )(C;E) \cong   \mcDy_n \mcC (\mcDy_m(C); \mcDy_m (E))  \ar[d]_-{\beta \mcDy (C;E)} \ar[r]^-{\mcDy_n\mcC(\gamma_C; \gamma_E)} & \mcDy_n \mcC(C;E) \ar[d]_-{\beta(C;E)} &\\
\mcC \mcDy_{n} \mcDy_m (C;E)\ar[d]_-{\mcC \gamma_{\mcDy}(C;E)} &  \mcC \mcDy_n (C;E)\cong \mcC (\mcDy_n(C); \mcDy_n(E)) \ar[d]_-{\mcC (\gamma_C; \gamma_E)} &\\
\mcC \mcDy_{n+m} (C;E) \cong  \mcC (\mcDy_{n+m} (C); \mcDy_{n+m} (E)) \ar[r]^-{\mcC(\gamma_C; \gamma_E)}  & \mcC(C;E) \ar[r]^{\mcC(C;E)} & \mcC(C;E)
}}$$
commute. Commutativity of the lower left square follows from the fact that $\gamma_C$ and $\gamma_E$ are $\mcDy$-algebra structure maps.
\end{proof}

Let $\Coder(E)$ denote the set of coderivations on the representation $(E,\omega)$. For cofree representations over cofree coalgebras we have the following result.

\begin{proposition} \label{prop:DerHom}
Let $X$ and $Y$ be bigraded modules, and let $\mcC$ be as above. Let $\mcC (X)$ be equipped with a coderivation $\partial_{\mcC(X)}$.
There is a bijection
$$\Coder(\mcC(X;Y)) \cong \Hom(\mcC(X;Y),Y).$$
Explicitly, the bijection is given by composing a coderivation with  $\xymatrix{\mcC(X;Y) \ar[r]^-{\epsilon} &Y}.$ 
To construct a  coderivation $\partial_f$ from a map $f\colon \mcC(X;Y)\ra Y$, set
$$\partial_f = d_{\mcC}(X;Y) +( \mcC \circ_{(1)} (f \vee \epsilon \partial_{\mcC(X)}))(\Delta_{(1)}(X;Y)),$$
where $\circ_{(1)}$ denotes the infinitesimal composite product of morphisms and $\Delta_{(1)}\colon \mcC(X;Y) \ra (\mcC \circ_{(1)} \mcC)(X;Y)$ denotes infinitesimal decomposition, see \cite[6.1.4]{LV12}. The map $f \vee (\epsilon \partial_{\mcC(X)})$ is either $f$ or $\epsilon \partial_{\mcC(X)}$ depending on whether the second copy of $\mcC$ in $(\mcC \circ_{(1)} \mcC)(X;Y)$ is decorated by an element in $Y$ or not.
\end{proposition}
\begin{proof}
Let $f\colon \mcC(X;Y) \ra Y$ be given and let $\gamma_{\mcC(X)} \colon \mcDy \mcC(X) \ra \mcC(X)$ correspond to $\partial_{\mcC(X)}$. 
Since $\mcDy$ is freely generated by $y$ we can define $\gamma_f \colon \mcDy   \mcC(X;Y)  \ra  \mcC(X;Y) $  by requiring that restricted to $\kfield y \otimes  \mcC(X;Y) \subset \mcDy  \mcC(X;Y)  $ it is given by
$$\xymatrix{
 \kfield y \otimes  \mcC(X;Y)  \ar[r]^-{\mcDy \Delta (X;Y)} & (\mcDy \mcC \mcC)(X;Y) \ar[r]^-{\beta \mcC(X;Y)}  &  \mcC(\mcDy \mcC(X); \mcDy  \mcC(X;Y) ) \ar[r]^-{\mcC(\epsilon\gamma_{\mcC(X)}; \bar{f})} &  \mcC(X;Y), 
}$$
where $\bar{f}\colon \mcDy   \mcC(X;Y)  \ra  Y$ resembles the sum of $f$ and the counit: It is defined by 
$$\bar{f}((y^j;c)(x_1,...,z,...,x_n))) = \begin{cases}
\epsilon(c)(x_1,...,z,...,x_n), & j=0,\\
f(c(x_1,...,z,...,x_n)),  & j=1,\\
0, & j>1.
\end{cases}
$$
We need to show that $\gamma_f$ corresponds to a coderivation.
We saw in the proof of Proposition~\ref{prop:CoderRepViaD} that (\ref{CoderRepViaD}) holds if it holds restricted to $\kfield y \otimes  \mcC(X;Y)$, and hence we only consider that case. First observe that 
$$\gamma_{\mcC(X)}(y) = d_{\mcC} + (\mcC \circ_{(1)} \epsilon\gamma_{\mcC(X)}(y)) \Delta_{(1)}$$
according to~\cite[6.3.8]{LV12}. Hence restricted to $\kfield y \otimes \mcC(X)$ the map $\gamma_{\mcC(X)}$ equals
$$\xymatrix{
\kfield y \otimes \mcC(X) \ar[r]^-{\mcDy \Delta (X)} & \mcDy \mcC \mcC(X) \ar[r]^-{\beta \mcC (X)} & \mcC\mcDy\mcC (X) \ar[r]^-{\mcC(\epsilon \gamma_{\mcC(X)})} & \mcC(X). 
}$$
We conclude that
$$ \mcC(\gamma_{\mcC(X)};\gamma_f) \colon   \mcC(\mcDy  \mcC(X); \mcDy  \mcC(X;Y)) \ra \mcC(\mcC(X); \mcC(X;Y))$$
can be written as
$$\xymatrix{
(\mcC \mcDy\mcC)(X;Y) \ar[r]^-{ \mcC \mcDy \Delta(X;Y)} &  (\mcC \mcDy  \mcC \mcC)(X;Y) \ar[r]^-{\mcC \beta \mcC(X;Y)} & ( \mcC \mcC \mcDy  \mcC )(X;Y) \ar[r]^-{ \mcC \mcC ( \epsilon \gamma_{\mcC(X)};  \bar{f})} &
 \mcC \mcC(X;Y) .
}$$
Hence we need to examine the diagram
$$\xymatrix{
y \mcC(X;Y) \subset \mcDy  \mcC(X;Y)  \ar[r]^-{\mcDy \Delta(X;Y)}\ar[dd]_-{\mcDy \Delta(X;Y)} & (\mcDy  \mcC\mcC)(X;Y) \ar[d]_-{\beta  \mcC(X;Y)} \\ 
 & (\mcC\mcDy \mcC)(X;Y) \ar[d]_-{\mcC \mcDy \Delta (X;Y)} \\
(\mcDy  \mcC \mcC)(X;Y) \ar[d]_-{\beta  \mcC(X;Y)}  & (\mcC \mcDy\mcC \mcC)(X;Y) \ar[d]_-{\mcC \beta \mcC(X;Y)}\\
(\mcC\mcDy \mcC)(X;Y)\ar[d]_-{\cong}  &   \mcC \mcC \mcDy\mcC(X;Y) \ar[d]_-{\cong} \\
 \mcC(\mcDy\mcC(X); \mcDy \mcC(X;Y))  \ar[d]_-{\mcC (\epsilon \gamma_{\mcC(X)}; \bar{f})} & \mcC \mcC(\mcDy \mcC(X); \mcDy \mcC(X;Y)) \ar[d]_-{\mcC \mcC(\epsilon \gamma_{\mcC(X)};\bar{f})}\\
\mcC(X;Y)\ar[r]^-{\Delta(X;Y)}& (\mcC \mcC)(X;Y)\\
}$$
That $\Delta$ commutes with the two lower vertical maps is clear. Using that $\beta$ is a distributive law and the coassociativity of $\Delta$ yields that 
$\gamma_f$ indeed defines a coderivation. One easily checks that $\gamma_f(y)$ coincides with $d_{\mcC} + (\mcC \circ_{(1)} ((\epsilon \partial_{\mcC(X)} ) \vee f))\Delta_{(1)}$ as claimed.
 
It remains to show that $\Hom(\mcC(X;Y) ,Y)$ and $\Coder(\mcC(X;Y))$ are in bijection.
It is clear that $\epsilon \partial_f=f$. Given a coderivation $v$, to see that $\partial_{\epsilon v}=v$, calculate that

\begin{align*}
(\mcC \circ_{(1)}( (\epsilon \partial_{\mcC(X)}) \vee(\epsilon v))) &\Delta_{(1)}(X;Y)+d_\mcC(X;Y) \\
 &= 
 \mcC(\epsilon(X);\epsilon(X;Y))(\mcC\circ^{''}(\partial_{\mcC(X)};v))\Delta(X;Y)\\
&\quad+\mcC(\epsilon(X);\epsilon(X;Y))d_\mcC(\mcC(X);\mcC(X;Y))\Delta(X;Y)\\
 &= \mcC(\epsilon(X);\epsilon(X;Y))\Delta(X;Y) v\\
&=
(\mcC\epsilon)(X;Y)\Delta(X;Y) v  =  v,
\end{align*}
where the expression
$\mcC(\epsilon(X);\epsilon(X;Y))\Delta(X;Y) v$
is interpreted as the composite
$$\xymatrix{\mcC(X;Y)\ar[r]^v& \mcC(X;Y)\ar[r]^<<<<<{\Delta(X;Y)}& 
(\mcC \mcC)(X;Y)\cong \mcC(\mcC(X);\mcC(X;Y))
\ar[rr]^-{\mcC(\epsilon(X);\epsilon(X;Y))}&& \mcC(X;Y)
}$$
and the expression
$(\mcC\epsilon)(X;Y)\Delta(X;Y) v$ 
is interpreted as the composite
$$\xymatrix{\mcC(X;Y)\ar[r]^v& \mcC(X;Y)\ar[r]^<<<<<{\Delta(X;Y)}& 
(\mcC \mcC)(X;Y)
\ar[rr]^-{(\mcC\epsilon)(X;Y)}&& \mcC(X;Y).
}$$
\end{proof}

Since we are interested in codifferentials we need to examine squares of coderivations. Recall that in the coalgebra case it is well known that the square of a coderivation of odd vertical degree is again a coderivation.

\begin{lemma}
Let $g \colon E \ra E$ and $\partial_C$ be coderivations of odd vertical degree. Then $g^2$ is a coderivation for $d_{\mcC}=0$ with respect to the coderivation $\partial_C^2$ on $C$, i.e.\  the following diagram commutes.
$$\xymatrix{
E\ar[d]_{\omega} \ar[r]^-{g^2} &E \ar[d]_{\omega}\\
\mcC (C;E) \ar[r]^*!/^2.5pt/{\labelstyle \mcC \circ'' (\partial_C^2;g^2)} &  \mcC(C;E)
}$$
\end{lemma}
\begin{proof}
One calculates that due to our assumptions on the degrees of the maps involved
$$ \omega g^2 = (\mcC \circ'' (\partial_C;g))^2 \omega.$$
A closer look at the definitions together with the degree hypothesis shows that 
$(\mcC \circ'' (\partial_C;g))^2$ maps an element  $z \in \mcC(n) \otimes \bigoplus_{i=1}^{n}(C^{\otimes i-1} \otimes E \otimes C^{\otimes n-i})$ to
$$
\sum_{j=1}^n  (\mcC(n) \otimes 1^{\otimes j-1} \otimes (\partial_C \vee g)^2 \otimes 1^{\otimes k-j}) (z),
$$
with $1$ denoting the identity on $C$ or $E$.
Since $(\partial_C \vee g)^2 = \partial_C^2 \vee g^2$ we find that
$$(\mcC \circ'' (\partial_C ; g))^2 = \mcC \circ'' (\partial_C^2 ; g^2).
$$
\end{proof}

\subsection{Representations via coderivations}
\label{subsec:modulestructures}

Let $\mcP$ be a Koszul operad. We already saw that $\mcP_{\infty}$-algebra structures on a vertical bicomplex $A$ with vertical differential $d_A$ are in bijection with the class of square-zero coderivations $\partial_{h+d_A \epsilon}$ induced by $h \colon \overline{\mcP^{\antishriek}}(A) \ra A$ and the internal differential $d_A$ on $A$. We will now prove a similar result for representations. For background on Koszul duality and the cobar construction we refer the reader to \cite{GK94} and \cite{LV12}.

For $M\in \BCv$ to be a representation of $A$ means that there is a morphism 
$$f_{\infty}\colon \mcP_{\infty}(A;M)  \ra M$$
of vertical bicomplexes satisfying certain properties. Since $\mcP_{\infty}= \Omega(\mcP^{\antishriek})$ is free this is equivalent to giving a map 
$$f \colon \overline{\mcP^{\antishriek}}(A;M) \ra M$$ 
of bidegree $(0,1)$ on the augmentation ideal of $\mcP^{\antishriek}(A;M)$ such that
$$d_M f + f  d_{\overline{\mcP}^{\antishriek}(A;M)} + f_{\infty} d_2s^{-1} =0, $$
with $d_{\overline{\mcP}^{\antishriek}(A;M)}$ the differential on $\overline{\mcP}^{\antishriek}(A;M)$ induced by $d_{\mcP^{\antishriek}}$, $d_A$ and $d_M$. Here $d_2$ denotes the twisting differential of the cobar construction and $s^{-1}\colon  \overline{\mcP}^{\antishriek}(A;M) \ra s^{-1}\overline{\mcP}^{\antishriek}(A;M)$ the desuspension map.\\

By Proposition~\ref{prop:DerHom} the map $ d_M \epsilon + f \colon \mcP^{\antishriek}(A;M) \ra M$ gives rise to a coderivation $\partial_{d_M\epsilon+f}$ on $\mcP^{\antishriek}(A;M)$.

\begin{proposition}\label{P:rep_coder} Given an arbitrary map $f \colon \overline{\mcP^{\antishriek}}(A;M) \ra M$ the coderivation $\partial_{d_M \epsilon+f}$ squares to zero if and only if $f$ is constructed from a $\mcP_{\infty}$-representation as above.
\end{proposition}
\begin{proof}
The results above yield that we only need to check under which conditions $\epsilon \partial_{d_M\epsilon+f}^2$ vanishes. We have
\begin{eqnarray*}
\epsilon \partial_{d_M\epsilon+f}^2 
& = & d_M\epsilon \partial_f + f (d_{ \mcP^{\antishriek}} + (\mcP^{\antishriek} \circ_{(1)} ((d_M\epsilon+f)\vee (d_A\epsilon + h))) (\Delta_{(1)}(A;M)))\\
& = & d_M f + f d_{ \mcP^{\antishriek}}  + f(\mcP^{\antishriek} \circ_{(1)} ((d_M\epsilon)\vee (d_A\epsilon))) 
(\Delta_{(1)}(A;M)) \\
&&+ f(\mcP^{\antishriek} \circ_{(1)} (f\vee  h)) (\Delta_{(1)}(A;M)).
\end{eqnarray*}
Note that 
$$
f (\mcP^{\antishriek} \circ_{(1)} (d_M\epsilon \vee d_A\epsilon))(\Delta_{(1)}(A;M))$$
equals the differential induced on $\overline{\mcP^{\antishriek}}(A;M)$ by $d_A$ and $d_M$. Since $f$ is only non-zero on the augmentation ideal we hence find that
$$\epsilon \partial_{d_M+f}^2= f d_{\overline{\mcP^{\antishriek}}(A;M)}  + d_Mf+f(\mcP^{\antishriek} \circ_{(1)} (f\vee h)) (\Delta_{(1)}(A;M)). $$
But
$$f(\mcP^{\antishriek}\circ_{(1)}(f\vee h)) (\Delta_{(1)}(A;M)) = f_{\infty}d_2 s^{-1}$$
and the result follows.
\end{proof}

\begin{remark}\rm
One could also state the result by saying that for a bigraded module $M$ a map $g\colon \mcP^{\antishriek}(A;M) \ra M$ of bidegree $(0,1)$ induces a square-zero coderivation on $\mcP^{\antishriek}(A;M) $  if and only if $(M, g\vert_M)$ viewed as a vertical bicomplex with differential $g\vert_M$ is a $\mcP_{\infty}$-representation of $A$ with structure map induced by $g\vert_{\overline{\mcP^{\antishriek}}(A;M)}$. The formulation above is purely a choice of making the role of the vertical differential on 
$M$ explicit to emphasize the category we work in rather than keeping it implicit.
\end{remark}

\subsection{Coderivations of $(\dAs)^{\antishriek}$-representations and representations of derived $A_{\infty}$-algebras}

In section \ref{sec:coalgebras} we saw how to 
describe $\As^{\antishriek}$-coalgebras and $(\dAs)^{\antishriek}$-coalgebras in terms of traditional conilpotent coalgebras. We will now determine what a $(\dAs)^{\antishriek}$-representation of a $(\dAs)^{\antishriek}$-coalgebra $C$ looks like. The results in this section as well as their proofs are analogous to the results for $(\dAs)^{\antishriek}$-coalgebras in \ref{subsec:dAsCoalgebras}. In particular it yields more insights to describe the structure on the suspension of a representation rather than the representation itself. 

\begin{remark}\rm
Before we concentrate on the derived case, let us consider representations of ordinary $A_{\infty}$-algebras. We know that conilpotent $\As^{\antishriek}$-coalgebras and  conilpotent coassociative coalgebras correspond to each other, and so do the notions of $\As^{\antishriek}$-coderivation and traditional coderivation. Recall that under this correspondence an $\As^{\antishriek}$-coalgebra $C$ corresponds to the traditional coalgebra $sC$.\\
For representations the same reasoning shows that $(E, \omega)$ is an $\As^{\antishriek}$-representation of $C$ if and only if $sE$ is a coassociative $sC$-bicomodule. One easily checks that $\As^{\antishriek}$-coderivations on $E$ coincide with coderivations of $sE$ as a bicomodule. 

In particular, for $sC=\overline{T}^c(sA)\cong s\As^{\antishriek}(A)$ equipped with a square-zero coderivation making $A$ an $A_{\infty}$-algebra we find that representations of $A$ correspond to codifferentials on the $\overline{T}^c(sA)$-bicomodule $T^c(sA) \otimes sM \otimes T^c(sA) \cong s\As^{\antishriek}(A;M)$.  
Hence we retrieve the notion  of $\emph{two-sided module over an $A_{\infty}$-algebra}$ given by Getzler and Jones \cite{GJ90}.
\end{remark}

\begin{proposition}
There is an equivalence between the category of $(\dAs)^{\antishriek}$-representations $E$ of $C$ and the category whose objects are $sC$-bicomodules 
$$(sE, \Delta^L\colon sE \ra sC \otimes sE, \Delta^R\colon sE \ra sE \otimes sC),$$ 
together with a map $f_{sE} \colon sE\ra sE$ of bidegree $(1,1)$ such that
$$ (f_{sC} \otimes sE)\Delta^L = \Delta^L f_{sE} =  (sC \otimes f_{sE}) \Delta^L $$
and
$$ (sE \otimes f_{sC})\Delta^R = \Delta^R f_{sE} = (f_{sE} \otimes sC) \Delta^R$$
and whose morphisms are bicomodule morphisms commuting with $f_{sE}$.
Under this equivalence a $(\dAs)^{\antishriek}$-coderivation of $E$ of bidegree $(0,1)$ corresponds to a 
coderivation of $sE$ as an $sC$-bicomodule of the same bidegree anti-commuting with $f_{sE}$.
\end{proposition}
\begin{proof}
We recalled that $E$ is a  $(\dAs)^{\antishriek}$-representation of $C$ if and only if $sE$ is a 
$\Lambda  (\dAs)^{\antishriek}$-representation of $sC$, hence we might as well determine what 
$\Lambda  (\dAs)^{\antishriek}$-representations are. Similar considerations hold for coderivations on these structures. So suppose 
$E'$ is a $\Lambda (\dAs)^{\antishriek}$-representations of $C'$. Let 
$$\rho \colon C' \ra \Lambda (\dAs)^{\antishriek}(C') \quad \text{and} \quad \omega \colon E' \ra \Lambda 
(\dAs)^{\antishriek}(C';E')$$
denote the structure maps and let
$$\rho^{i,n} \colon \xymatrix{ C' \ar[r]^-{\rho} & \Lambda (\dAs)^{\antishriek}(C') \ar@{>>}[r] & \kfield 
\alpha_{in}\otimes C'^{\otimes n} \ar[r]^-{\cong} & C'^{\otimes n}}$$
and
$$\omega^{i,n} \colon \xymatrix{ E' \ar[r]^-{\omega} & (\dAs)^{\antishriek}(C';E') \ar@{>>}[r] & \kfield 
\alpha_{in}\otimes (\bigoplus_{a+b+1=n}C'^{\otimes a} \otimes E' \otimes C'^{\otimes b})  \\
&\ar[r]^-{\cong} & \bigoplus_{a+b+1=n}C'^{\otimes a} \otimes E' \otimes C'^{\otimes b}}$$
be the projections of the structure maps to the indicated components. Here $i \geq 0$ and $n \geq 1$ with $\rho^{0,1}$ and $\omega^{0,1}$ equal to the identity. 

Spelling out the coassociativity condition for $\omega$ in terms of these projections yields the condition that
\begin{equation} \label{eq:expdAscoass}
((\rho/\omega)^{i_1, k_1} \otimes \cdots\otimes (\rho/\omega)^{i_n,k_n})\omega^{i,n} 
= (-1)^{ \sigma}\omega^{i_1+\cdots+i_n+i, k_1+\cdots+k_n}
\end{equation}
where $\sigma= i(k_1+\cdots+k_n+n) + \sum_{1 \leq x<y \leq n}(i_x k_y +i_y k_x)$,
for all $i,i_1,...,i_n \geq 0$ and $n,k_1,...,k_n \geq 1$, with $(\rho/\omega)^{r,s}$ denoting $\rho^{r,s}$ or $\omega^{r,s}$ depending on the input; see  Corollary~\ref{cor:lambda}.
In particular
\begin{equation}\label{eq:expdAscoder1}
((\rho/\omega)^{0,2} \otimes  1) \omega^{0,2} = (1 \otimes (\rho/\omega)^{0,2}) \omega^{0,2}
\end{equation}
with $1$ denoting either the identity on $C'$ or $E'$, because both terms coincide with $\omega^{0,3}$ and
\begin{equation}\label{eq:expdAscoder2}
((\rho/\omega)^{1,1} \otimes 1) \omega^{0,2} =  \omega^{0,2} \omega^{1,1} = (1 \otimes (\rho/\omega)^{1,1}) \omega^{0,2},
\end{equation}
because all of these compositions are equal to $-\omega^{1,2}$.
Hence $sE$ is an $sC$-bicomodule with a map $f_{sE}= s\omega^{1,1}$ having the properties claimed above. One also sees that 
\begin{equation}\label{eq:omegars}
\omega^{r,s} =  (-1)^{r(s+1)}\omega^{0,s} (\omega^{1,1})^r
\end{equation}
with 
$$\omega^{0,s} = ((\rho/\omega)^{0,2} \otimes 1^{\otimes s-2})((\rho/\omega)^{0,2} \otimes 1^{\otimes s-3}) ...((\rho/\omega)^{0,2} \otimes 1) \omega^{0,2}$$ denoting iterated applications of $(\rho/\omega)^{(0,2)}$. Hence $\omega^{0,2}$ and $\omega^{1,1}$ determine $\omega$ and one can calculate that (\ref{eq:omegars}) together with (\ref{eq:expdAscoder1}) and (\ref{eq:expdAscoder2}) yields (\ref{eq:expdAscoass}).\\
Since $\alpha_{0,2}$ has vertical degree $0$ and $\alpha_{1,1}$ has vertical degree $1$ a $\Lambda(\dAs)^{\antishriek}$-coderivation is a map that is a coderivation with respect to the copoduct $\omega^{0,2}$ and anticommutes with $\omega^{1,1}$.
\end{proof}

Applying this to $C= (\dAs)^{\antishriek}(A)$ and $E= (\dAs)^{\antishriek}(A;M)$ we find the following.

\begin{proposition}
The $(\dAs)^{\antishriek}$-representation $(\dAs)^{\antishriek}(A;M)$ corresponds to the $\kfield[x]\otimes \overline{T}^c(sA)$-bicomodule structure on  $\kfield[x] \otimes T^c(sA) \otimes sM \otimes T^c(sA)$ given by 
\begin{eqnarray*}
&&\Delta^L(x^i \otimes (sa_1 , ... , sa_{j-1} , sm , sa_{j+1} , ... , sa_n))\\
&& 
\qquad= \sum_{k=1}^{j-1} \sum_{r+s=i}  (-1)^{\epsilon}(x^r \otimes (sa_1 ,..., sa_k)) \otimes (x^s \otimes (sa_{k+1} ,... , sm ,..., sa_n)),\\
&&\Delta^R (x^i \otimes (sa_1 , ... ,s a_{j-1} , sm , sa_{j+1} , ... , sa_n))\\
&&
\qquad =\sum_{k=j}^{n} \sum_{r+s=i}  (-1)^{\epsilon}(x^r \otimes (sa_1 ,... , sm ,..., sa_k)) \otimes (x^s \otimes (sa_{k+1} ,... , sa_n)),
\end{eqnarray*}
with $\epsilon = r(n+k)+(s,s)(\vert a_1\vert +...+ \vert a_k \vert),$
together with the map
\begin{eqnarray*}
f &\colon& \kfield[x]\otimes T^c(sA) \otimes sM \otimes T^c(sA) \ra \kfield[x]\otimes T^c(sA) \otimes sM \otimes T^c(sA), \\
&& x^i \otimes (sa_1 , ... , sm ,... sa_n) \mapsto (-1)^{n+1} x^{i-1} \otimes (sa_1 , ... , sm ,... sa_n )
\end{eqnarray*}
with $x^{-1}=0$.
\end{proposition}

\begin{proposition}\label{prop:splittingUpCoderivations}
Let $d$ be a coderivation of $\Lambda (\dAs)^{\antishriek}(sA)$ giving rise to a family  $d_i$  of coderivations on $\overline{T}^c(sA)$ as discussed in Example~\ref{ex:FreedAsCoalg}. Giving a $\Lambda(\dAs)^{\antishriek}$-coderivation $g$ on $\Lambda(\dAs)^{\antishriek}(sA;sM)$ is equivalent to specifying a family of maps 
$$g_j \colon T^c(sA) \otimes sM \otimes T^c(sA) \ra T^c(sA) \otimes sM \otimes T^c(sA), \quad j \geq 0,$$
of bidegree $(-j,1-j)$ such that $g_j$ is a $\overline{T}^c(sA)$-bicomodule coderivation with respect to $d_j$.
\end{proposition}
\begin{proof}
Denote by $g_{i,j}$ the component
$$ \kfield x^{i} \otimes T^c(sA) \otimes sM \otimes T^c(sA) \ra \kfield x^{j} \otimes T^c(sA) \otimes sM \otimes T^c(sA)$$
of $g$. Since $g$ has to anti-commute with $f$ we see that
$$f^j g_{i,j} = \begin{cases}
(-1)^j g_{i-j,0} f^j, & i\geq j,\\
0, & j>i
\end{cases}$$
and hence that $g$ is completely determined by the maps $g_{r,0}$. 
Define $g_r $ by
$$g_r(sa_1,...,sa_{i-1},sm,sa_{i+1},...,sa_n) = (-1)^{rn}g_{r,0}(x^r \otimes (sa_1,...,sm,...,sa_n)).$$
Then the $g_r$ are bicomodule coderivations if and only if $g$ is a $\Lambda (\dAs)^{\antishriek}$-coderivation.
\end{proof}

Applying Proposition \ref{P:rep_coder} to the case where $\mathcal P=\dAs$ we get that a representation $M$ of a derived
$A_\infty$-algebra $A$ is entirely determined by a square-zero coderivation $g$ of the representation
$(\dAs)^{\antishriek}(A;M)$ of the $(\dAs)^\antishriek$-coalgebra $(\dAs)^{\antishriek}(A)$
(endowed itself with the square-zero
derivation $d$ defining the $A_\infty$-algebra structure on $A$).
In Proposition \ref{prop:splittingUpCoderivations} we have described explicitly the coderivation $g$. In the next theorem, we characterize the square-zero
coderivations.

\begin{theorem}
Let $A$ be a $dA_{\infty}$-algebra, and let $h_i \colon \overline{T}^c(sA) \ra \overline{T}^c(sA)$ be the corresponding coderivations making $\overline{T}^c(sA)$ a twisted chain complex as discussed in Proposition~\ref{prop:dAInfAlgAsTw}. Then endowing a bigraded $\kfield$-module $M$ with the structure of a $dA_{\infty}$-representation of $A$ is equivalent to giving maps 
$$g_i \colon T^c(sA) \otimes sM \otimes T^c(sA) \ra T^c(sA) \otimes sM \otimes T^c(sA), \quad i\geq 0,$$
of bidegree $(-i,1-i)$ such that
\begin{itemize}
\item the $g_i$ make $T^c(sA) \otimes M \otimes T^c(sA)$ a twisted chain complex,
\item for all  $i\geq 0$ the map $g_i$ is a bicomodule coderivation with respect to $h_i$.
\end{itemize}
\end{theorem}
\begin{proof}
We saw how to construct the maps $g_i$ from a coderivation $g \colon  \Lambda(\dAs)^{\antishriek}(A;M) \ra \Lambda(\dAs)^{\antishriek}(A;M)$ in the proof of Proposition~\ref{prop:splittingUpCoderivations}. The $g_i$ define a twisted chain complex if and only if for all $u \geq 0$ and all $(sa_1 ,... , sm , ..., sa_n) \in T^c(sA) \otimes sM  \otimes T^c(sA)$
\begin{eqnarray*}
0 
& = & 
\sum_{i+p=u} (-1)^i g_i g_p (sa_1,..., sm,..., sa_n)\\
& = & 
\sum_{i+p=u} (-1)^{i+pn} g_i g_{p,0} (x^p \otimes (sa_1,..., sm,..., sa_n))\\
& = & 
\sum_{i+p=u} (-1)^{i+pn+i(n+1)} g_i g_{p,0} f^i(x^{p+i} \otimes (sa_1,..., sm,..., sa_n))\\
& = & 
\sum_{i+p=u} (-1)^{pn+i(n+1)} g_i f^i g_{p+i,i} (x^{p+i} \otimes (sa_1,..., sm,..., sa_n)).\\
\end{eqnarray*}
But $g_i f^i = (-1)^i g_{i,0}$ on $\kfield x^i\otimes T^c(sA)\otimes sM\otimes T^c (sA)$, hence the $g_i$ yield a twisted chain complex if and only if
$$0= \sum_{i+p=u} (-1)^{un} g_{i,0} g_{p+i,i}.$$
Hence the projection of $g^2$ to $\kfield x^0\otimes T^c(sA) \otimes sM \otimes T^c(sA)$ is zero, and Proposition~\ref{prop:DerHom} yields that $g^2=0$ in general.
\end{proof}

\begin{remark}\rm
In \cite[6.2]{Sag10} Sagave defines a module over a $\dAs_{\infty}$-algebra $A$ as a bigraded $\kfield$-module $M$ such that $sM\otimes T^c(sA)$ is a twisted chain complex whose $i$-th structure map $g_i$ is a right $\overline{T}^c(sA)$-coderivation with respect to $h_i$. The operadic notion of representation hence yields a two-sided variant of Sagave's definition.
\end{remark}

\section{New example of a derived $A_\infty$-algebra}
\label{sec:examples}

In this section, we will use a family of examples of finite dimensional $A_\infty$-algebras given by Alloca and Lada in~\cite{Allocca10} in 
order to build a new example of a 3-dimensional derived $A_\infty$-algebra.

\subsection{Examples of finite dimensional $A_\infty$-algebras}

Alloca and Lada give in  \cite{Allocca10} a family of examples of $A_\infty$-algebras. Taking a subalgebra, one gets the following result as a corollary of~\cite[Theorem 2.1]{Allocca10}. Here, the sign conventions for  $A_\infty$-algebras
are those of Loday-Vallette.

\begin{proposition}\label{Aex}
The free graded $\kfield$-module $V$ spanned by $x$ of degree $0$ and $y$ of degree $1$ is an $A_\infty$-algebra 
with $\kfield$-linear structure maps satisfying:
\begin{align*}
m_1(x) & = y,\\
m_n(x \otimes y^{\otimes k} \otimes x \otimes y^{(n-2)-k} )&= (-1)^ks_n x, & \textrm{ for } 0 \leq k \leq n-2,\\
m_n(x \otimes y^{n-1}) &= s_{n+1}y, &
\end{align*}
where $s_n = (-1)^{(n+1)(n+2)/2}$, and $m_n(z)= 0$ for any $n$ and any basis element $z \in V^{\otimes n}$  not listed above. \qed
\end{proposition}
\begin{remark}\rm
 If we modify the above example so that $m_1 = 0$, but everything else is unchanged, then $V$ is still an 
 $A_\infty$-algebra. That is, we can construct a minimal example from the one above,  where we recall that a minimal  
$A_\infty$-algebra $A$ is an $A_\infty$-algebra such that $m_1=0$.
\end{remark}
\sline

\subsection{Example of a derived $A_\infty$-algebra}

We describe an example of a derived $A_\infty$-structure on a rank 3 free bigraded $\kfield$-module $V$ spanned by $u,v,w$ where
$|u| = (0,0), |v| = (-1,0),$ and $|w| = (0,1)$. 

\smallskip

Note that if $V$ is as above, the bidegree $(-k,l)$ of an element $z\in V^{\otimes j}$ satisfies $0\leq k,0\leq l$ and $k+l\leq j$. 
Since the structure map $m_{in}:V^{\otimes n}\rightarrow V$ is of bidegree
$(-i,2-i-n)$, the element $m_{in}(z)$ has bidegree $(-k-i,2-i-n+l)$. This has the following consequence.

\begin{proposition} If the bigraded $\kfield$-module $V$ as above is endowed with a derived $A_\infty$-structure then,
for reasons of bidegree,  $m_{in}(z)$ with $z \in V^{\otimes n}$ can be potentially non-zero only if
$0\leq i\leq 1$. Furthermore, letting $z=x_1\otimes\cdots\otimes x_n$ where each $x_l$ is one of the basis
elements of $V$, we have the following.
\begin{enumerate}
\item If $m_{0n}(z)\not=0$, then there exist $i\not=j$ such that  $x_k=w$ for $k\not\in\{i,j\}$ and
 $(x_i,x_j)\in\{(u,u),(u,w),(w,u),(u,v),(v,u)\}$.
\item If $m_{1n}(z)\not=0$,   then there exists $i$ such that $x_i=u$ and $x_k=w$ for $k\not=i$.\qedhere
\end{enumerate}\qed
\end{proposition}		

\begin{proposition}Let $V$ be the  rank 3 free bigraded $\kfield$-module as above.
Then $V$ is endowed with the following derived $A_\infty$-structure.
For $n\geq 2$,  we let
\begin{align*}
m_{0n}(u \otimes w^{\otimes k} \otimes u \otimes w^{\otimes (n-2)-k} )&= (-1)^ks_nu, & \text{ for } 0 \leq k \leq n-2,\\
m_{0n}(u \otimes w^{\otimes n-1}) &= s_{n+1}w,   \\
m_{0n}(u \otimes w^{\otimes n-2} \otimes v  )&= (-1)^{n-2}s_n v,
\end{align*}
and for $n\geq 1$, we let
\begin{align*}
m_{11}(u) &= v, \\
m_{1n}(u \otimes w^{ \otimes n-1}) &= s_{n+1}v,\\
\end{align*}
where $s_n = (-1)^{(n+1)(n+2)/2}$ and we let $m_{ij}(z)= 0$ for any $i,j$ and for any basis element $z \in V^{\otimes j}$  not covered by the cases above. 
\end{proposition}

\begin{proof} The proof is just a computation. We will not give full details, but we supply enough
 ingredients so that the computation can be carried out quickly.

Note that to check that $V$ is a derived $A_\infty$-algebra we only need to check that, for $l \geq 1$ and 
$z\in V^{\otimes l+1}$, the following three conditions hold.

\begin{align*}
\label{rels}
 \sum_{j+q = l+1} m_{0j} \star m_{0q} (z)= 0, \\
 \sum_{j+q = l+1} (m_{0j} \star m_{1q} + m_{1j} \star m_{0q})(z) = 0,\\
  \sum_{j+q = l+1} m_{1j} \star m_{1q}(z) = 0,
\end{align*}
with the $\star$-product defined in the formula (\ref{def:star}) of the appendix.

We consider the three relations in turn, outlining the checking required for each.
\medskip

\textbf{Relation I}\qquad $\sum_{j+q = l+1} m_{0j} \star m_{0q}(z) = 0$.
\sline

Let $V_0 =\langle u, w \rangle$ be the subspace of $V$ spanned by the elements of bidegree $(0,r)$, for $r\in\Z$.  If $V$ is a  derived $A_\infty$-algebra, then $V_0$ is an $A_\infty$-algebra. As a consequence checking the equation on tensors $z$ not containing $v$ is equivalent to checking that $V_0$ is an $A_\infty$-algebra. This is true by Proposition \ref{Aex}. 

It  remains to check the equation on tensors containing $v$. 
For terms containing at least one $v$, $m_{0j}(1^{\otimes *} \otimes m_{0q} \otimes 1^{\otimes *})$ is possibly non-zero only on tensors of the form 
$$u \otimes w^{\otimes k} \otimes u \otimes w^{\otimes l-k-3} \otimes v,  \textrm{  for } 0 \leq k \leq l-3, \text{ where } j+q = l+1,$$ 
and a sign computation shows that the expression vanishes on those terms.

\sline

\textbf{Relation II}\qquad  $\sum_{j+q = l+1} (m_{0j} \star m_{1q} + m_{1j} \star m_{0q}) (z)= 0$.
\sline

This case is similar to the previous one;  $m_{0j}(1^{\otimes *} \otimes m_{1q} \otimes 1^{\otimes *})+m_{1j}(1^{\otimes *} \otimes m_{0q} \otimes 1^{\otimes *})$ is possibly non-zero only on tensors of the form 

$$u \otimes w^{\otimes k} \otimes u \otimes w^{\otimes l-k-2},  \textrm{  for } 0 \leq k \leq l-2, \text{ where } j+q = l+1.$$

\textbf{Relation III}\qquad $\sum_{j+q = l+1} m_{1j} \star m_{1q}(z) = 0$.

\sline 

Since $m_{1n}$ takes values zero or $\pm v$ on basis elements and since $m_{1n}$ applied to a tensor containing a $v$ vanishes, it follows that
 $\sum_{j+q = l+1} m_{1j} \star m_{1q} (z)=0$.  \end{proof}

\begin{remark}\rm
In this example, we have $m_{01} = 0$; that is, we have a minimal derived $A_\infty$-algebra.

For bidegree reasons, the only alternative would be letting $m_{01}(u)$ be (some multiple of) $w$.
However, modifying the above example so that $m_{01}(u) = w$, with everything else unchanged, does
not give a  derived $A_\infty$-algebra. A direct computation shows that we would have
	\[
	\sum_{j+q = 4} (m_{0j} \star m_{1q}+m_{1j} \star m_{0q})(u \otimes w \otimes u) = v \neq 0
	\]
 and 
	\[
	\sum_{j+q = 4} (m_{0j} \star m_{1q}+m_{1j} \star m_{0q})(u \otimes u \otimes w) = -v \neq 0.
	\]

\sline
On the other hand, if we `truncate' the above example, setting $m_{ij} =0$ for $i+j \geq 3$, then it can be checked, 
using SAGE, that we get a bidga, both in the case with $m_{01} = 0$ and also in the case where we modify the example so that $m_{01}(u) = w$.
\end{remark}

\section{Appendix: sign conventions}

In this appendix, we compare different sign conventions relating to 
derived $A_\infty$-algebras. In the special case of
$A_\infty$-algebras such comparisons have been mentioned in the literature.

\subsection{Different conventions for derived $A_\infty$-algebras}

We recall that a derived $A_\infty$-structure on $A$
consists of $\kfield$-linear maps
$m_{ij}: A^{\otimes j} \longrightarrow A$
of bidegree $(-i,2-i-j)$ for each $i \ge 0$, $j\ge 1$,  satisfying the equation (\ref{dobjectequation})
of Definition~\ref{defn:dAinftystr}:
\begin{equation*}
\sum\limits_{\substack{ u=i+p, v=j+q-1 \\ j=1+r+t}} (-1)^{rq+t+pj}
m_{ij} (1^{\otimes r} \otimes m_{pq} \otimes 1^{\otimes t}) = 0.
\end{equation*}
Consequently the family of maps $m_{0n}$ satisfies the equation
\begin{equation*}
\sum\limits_{\substack{v=j+q-1 \\ j=1+r+t}} (-1)^{rq+t}
m_{0j} (1^{\otimes r} \otimes m_{0q} \otimes 1^{\otimes t}) = 0,
\end{equation*}
which is the sign convention of Getzler and Jones in~\cite{GJ90}.
In the definition of derived $A_\infty$-algebra if we pick the  generators 
	\[
	\widetilde m_{ij}=(-1)^{\frac{j(j-1)}{2}}m_{ij}
	\]
one gets
\begin{equation*}
\sum\limits_{\substack{ u=i+p, v=j+q-1}}\widetilde m_{ij} \star \widetilde m_{pq}=0,
\end{equation*}
with 
\begin{equation}\label{def:star}
\widetilde m_{ij} \star \widetilde m_{pq} =\sum\limits_{k=1}^j (-1)^{i+j+(q-1)(k+j)+p(j-1)}
\widetilde m_{ij} \circ_k \widetilde m_{pq}
\end{equation}
The family $\widetilde m_{0n}$ satisfies
\begin{equation*}
\sum\limits_{\substack{ u=i+p, v=j+q-1}}\sum\limits_{k=1}^j (-1)^{vq+k(q-1)}
\widetilde m_{0j} \circ_k \widetilde m_{0q}=0,
\end{equation*}
which is the original definition of $A_\infty$-algebras by Stasheff \cite{Sta63}.

\subsection{Different sign conventions for the cooperad $(\dAs)^{\antishriek}$}

For any graded cooperad $\mathcal{C}$, if one has elements $a_{uv}\in {\mathcal{C}}(v)$ satisfying
$$\Delta(a_{uv})=\sum\limits_{q_1+\cdots+q_j=v} (-1)^{X(I)} a_{ij}; a_I$$
with $a_I=a_{p_1q_1}\otimes\cdots\otimes a_{p_jq_j}$, then setting
$\widetilde a_{uv}=(-1)^{\frac{v(v-1)}{2}} a_{uv}$, one gets

$$\Delta(\widetilde a_{uv})=\sum (-1)^{X(I)}(-1)^{\phi(I)} \widetilde a_{ij}; \widetilde a_I,$$
where $\phi(I)$ is obtained modulo 2 as
\[
\phi(I)=\frac{1}{2}\left(j(j-1)+(\sum_k q_k)((\sum_l q_l)-1)+\sum_k q_k^2-\sum_l q_l \right)=
\sum_{k=1}^{j-1} k+\sum_{k<l} q_kq_l.
\]

Recall that the cooperad $(\dAs)^{\antishriek}$ has generators  $\mu_{uv}$ of bidegree $(-u,1-u-v)$ with structure map given by
\begin{equation*}
    \Delta(\mu_{uv})=\sum\limits_{i+p_1+\cdots+p_j=u\atop{q_1+\cdots+q_j=v}} 
                    (-1)^{X\left((p_1,q_1), \dots, (p_j,q_j)\right)}
                        \mu_{ij};\mu_{p_1q_1}\otimes\cdots\otimes \mu_{p_j q_j},
\end{equation*}
with
$  X\left((p_1,q_1), \dots, (p_j,q_j)\right) = \sum\limits_{1\leq k<l\leq j} \Big(p_k+ q_k(p_l + q_l+1)     \Big)$ (see formula (4) in \cite{LRW13}). 
Consequently the bigraded $\kfield$-module generated by the family $(\widetilde\mu_{0v})_{v\geq 0}$ is a subcooperad of $(\dAs)^{\antishriek}$  and satisfies
\begin{equation*}
    \Delta(\widetilde\mu_{0v})=\sum\limits_{q_1+\cdots+q_j=v}
                    (-1)^{X'(q_1,\ldots,q_j)}
                        \widetilde\mu_{0j};\widetilde\mu_{0q_1}\otimes\cdots\otimes \widetilde\mu_{0 q_j},
\end{equation*}
with
\[
X'(q_1,\ldots,q_j)\equiv\sum\limits_{1\leq k<l\leq j} (q_k(q_l+1)+q_kq_l )+\sum_{k=1}^{j-1} k\equiv\sum\limits_{k=1}^{j-1} (q_k(k+j)+k)
\equiv\sum\limits_{k=1}^{j-1}(q_k+1)(k+j),
\]
where the computation is performed modulo 2.
We recover the signs obtained by Loday and Vallette in \cite{LV12} in their definition of the cooperad $\As^{\antishriek}$.

Note that if we choose $\widetilde\mu_{uv}$ as generators for the cooperad  $(\dAs)^{\antishriek}$, the structure map is given by
\begin{equation*}
    \Delta(\widetilde\mu_{uv})=\sum\limits_{i+p_1+\cdots+p_j=u\atop{q_1+\cdots+q_j=v}} 
                    (-1)^{X'\left((p_1,q_1), \dots, (p_j,q_j)\right)}
                        \widetilde\mu_{ij};\widetilde\mu_{p_1q_1}\otimes\cdots\otimes \widetilde\mu_{p_j q_j},
\end{equation*}
where
$X'\left((p_1,q_1), \dots, (p_j,q_j)\right)=\sum\limits_{k=1}^{j-1} (p_k+q_k+1)(k+j)+\sum_{k<l} (q_kp_l)$.

\subsection{Description of the cooperad $\Lambda(\dAs)^{\antishriek}$}

The notion of suspension of a cooperad was explained in Section~\ref{subsec:coopsusp}.
Here we establish the sign conventions for the cooperad structure of $\Lambda\mcC$ for a cooperad $\mcC$ in $\BCv$.

\begin{proposition}\label{sign:lambda}
Let $\mcC$ be a cooperad in $\BCv$. Then $\Lambda \mcC$ is the cooperad with
$$(\Lambda \mcC)(n) = s^{1-n} \mcC(n).$$
The cooperad structure of $\Lambda \mcC$ is given by
$$ s^{1-n}c \mapsto \sum (-1)^{\sum_{k=1}^j\vert s^{l_k+1} \vert \vert s^{1-j} c' \vert + \sum_{k=2}^j \sum_{l=1}^{k-1} \vert s^{l_k+1}\vert \vert sc''_l \vert}s^{1-j}c';s^{1-l_1}c''_1,...,s^{1-l_j}c''_j,$$
where the decomposition map of $\mcC$ maps $c \in \mcC(n)$ to the sum $\sum c'; c''_1,...,c''_j$ with $c' \in \mcC(j)$ and  $c''_i \in \mcC(l_i)$ for $1\leq i \leq j$. 
\end{proposition}

\begin{proof}We explain the algorithm for distributing $s^{1-n}$ over the different tensor products.

Firstly put $s^{1-j}$ in front of $c'$. This operation is sign free.

Secondly, distribute one $s$, that will be in front of $c''_k$, for $k$ going from 1 to $j$: first $s$ jumps over $s^{1-j}c'$ and is placed in 
front of $c''_1$; second $s$ jumps over
$s^{1-j}c'\otimes sc''_1$ and is placed in front of $c''_2$. The sign  involved is obtained as $(-1)^x$ where $x$ mod $2$ is
$\sum_{k=1}^j |s||s^{1-j}c'|+\sum_{k=2}^j |s|(\sum_{l=1}^{k-1}|sc''_l|).$

Finally, for $k$ going down from $j$ to $1$ distribute $s^{-l_k}$ over $s^{1-j}c'\otimes sc''_1\otimes\cdots\otimes sc''_{k-1}$.
The sign involved is  obtained as $(-1)^x$ where $x$ mod $2$ is
$\sum_{k=1}^j |s^{l_k}||s^{1-j}c'|+\sum_{k=2}^j |s^{l_k}|\sum_{l=1}^{k-1}|sc''_l|.$
\end{proof}

\begin{corollary}\label{cor:lambda}The cooperad $\Lambda(\dAs)^{\antishriek}$ has generators $\alpha_{uv}$ of bidegree $(-u,-u)$ and the 
cooperad structure is given by
$$\Delta(\alpha_{uv})=\sum\limits_{i+p_1+\cdots+p_j=u\atop{q_1+\cdots+q_j=v}} 
                    (-1)^{i(v+j)+\sum\limits_{1\leq k<l\leq j} p_kq_l+q_kp_l}
                    \alpha_{ij};\alpha_{p_1q_1}\otimes\cdots\otimes \alpha_{p_j q_j},$$
\end{corollary}

Note that if $u=0,i=0,p_k=0$ one gets exactly the cooperad $\As^*$.

\begin{proof}This is a short sign computation. Let $I=((p_1,q_1),\ldots,(p_j,q_j))$ and let $S(I)$ be the sum such that $(-1)^{S(I)}$ is the sign 
defined in Proposition \ref{sign:lambda}. 
We recall that $\alpha_{uv}=s^{1-v}\mu_{uv}$ and that $s\mu_{uv}$ has bidegree $(-u,-u-v)$. Computing mod 2, one gets
\begin{multline*}
X(I)+S(I)\equiv\sum\limits_{1\leq k<l\leq j} \Big(p_k+ q_k(p_l + q_l+1)\Big) +\sum\limits_{k=1}^j(q_k+1)i+\sum\limits_{1\leq k<l\leq j}(p_k+q_k)(q_l+1)\\
\equiv i(v+j)+\sum\limits_{1\leq k<l\leq j} q_kp_l+p_kq_l.
\end{multline*}
\end{proof}

\begin{example}\label{Ex:signs}\rm
 As an example one has
\begin{eqnarray*}
  \Delta(\alpha_{12})&=&\alpha_{12};\alpha_{01}\otimes\alpha_{01}-\alpha_{11};\alpha_{02}-
  \alpha_{02};(\alpha_{11}\otimes\alpha_{01}+\alpha_{01}\otimes\alpha_{11})+\alpha_{01};\alpha_{12}, \\
  \Delta(\alpha_{ij})&=&(-1)^{i(j+1)} \alpha_{i1};\alpha_{0j}+\text{other terms}.
 \end{eqnarray*}
\end{example}

\providecommand{\bysame}{\leavevmode\hbox to3em{\hrulefill}\thinspace}
\providecommand{\MR}{\relax\ifhmode\unskip\space\fi MR }
\providecommand{\MRhref}[2]{%
  \href{http://www.ams.org/mathscinet-getitem?mr=#1}{#2}
}
\providecommand{\href}[2]{#2}

\end{document}